\documentclass{amsart}
\usepackage{amssymb}
\usepackage{amsmath}
\usepackage{subfigure}
\usepackage{epsfig}
\usepackage[all]{xy}

\newtheorem{theorem}{Theorem}[section] 
\newtheorem{lemma}[theorem]{Lemma}     

\newtheorem{proposition}[theorem]{Proposition}

\theoremstyle{definition}
\newtheorem{example}[theorem]{Example}

\title[Branches on division algebras]
{Branches on division algebras}

\author{Manuel Arenas}
\author{Luis Arenas-Carmona}

\newcommand\Qt{\mathbb Q}
\newcommand\algeb{\mathfrak{A}}
\newcommand\Da{\mathfrak{D}}

\newcommand\Ha{\mathfrak{H}}
\newcommand\oink{\mathcal O}
\newcommand\matrici{\mathbb{M}}

\newcommand\ad{\mathbb{A}}
\newcommand\enteri{\mathbb Z}

\newcommand\bbmatrix[4]{\left(\begin{array}{cc}#1&#2\\#3&#4\end{array}\right)}
\newcommand\lbmatrix[4]{\textnormal{\scriptsize{$\left(\begin{array}{cc}#1&#2\\#3&#4
\end{array}\right)$}}\normalsize}

\begin{document}
\maketitle

\begin{abstract}
We describe the set of maximal orders in a 2-by-2 matrix algebra over a non-commutative local division algebra $B$
containing a given suborder, for certain important families of such suborders, including rings of integers of division subalgebras
of $B$ or most maximal semisimple commutative subalgebras.
\end{abstract}

\section{Introduction}

Let $B$ be a local division algebra, i.e., an $n^2$-dimensional central division algebra over a local field $K$. By an order $\Ha\subseteq\matrici_2(B)$ we mean an $\oink_K$-order, where $\oink_K$ is the ring of integers in $K$. For any such order $\Ha$, let $S(\Ha)$ be the set of maximal orders containing $\Ha$. Let $\mathfrak{T}=\mathfrak{T}(B)$ be the Bruhat-Tits tree \cite{trees} for $PSL_2(B)$, with its vertices identified with the maximal orders in $\matrici_2(B)$.  We call it the BT-tree in all that follows. 
Let $\mathfrak{S}(\Ha)$ be the maximal subgraph whose vertex set is $V\big(\mathfrak{S}(\Ha)\big)=S(\Ha)$.
The subgraph $\mathfrak{S}(\Ha)$ is called the branch of $\Ha$.

The branch $\mathfrak{S}(\Ha)$ is well understood when $B=K$ \cite{eichler2}.
This knowledge is critical in solving the selectivity problem 
\cite{FriedmannQ}, \cite{vigneras},
determining whether every maximal order, in a given global central simple algebra, 
contains an isomorphic copy of a given suborder $\Omega$. This is a problem of interest in differential geometry, 
since spectral properties of some hyperbolic varieties are reduced to selectivity
problems. More specifically, this theory was applied in some early construction of
isospectral, non-isometric Riemanian varieties, an hyperbolic analog to the famous question \emph{Can we hear the shape of a drum?} \cite{vigneras2}, \cite{lino2}, \cite{lino4}.
 There are also numerous questions on the algebraic or arithmethic structure
of maximal orders, or commutative order, that reduce to selectivity problems \cite{cyclic}.
Explicit descriptions of the branch $\mathfrak{S}(\Ha)$ can also be aplied to the study
of quotient graphs in the quaternionic case \cite{rouidqo}, \cite{cqqgvro}.

A precise description of the branch $\mathfrak{S}(\Ha)$, or its generalization in terms of buildings,
 for an arbitrary order $\Ha$, in an arbitrary local central simple algebra $A$, would be equally useful.
Unfortunately, a full description along the lines of the
three types of branches listed in \cite{eichler2} (or Prop. \ref{p21} below)  seems extremely unlikely in 
the light of some results presented in this work. A full description for some natural
families of suborders seems, however, within reach. In fact, the techniques developed
in this work allow us to compute the branch of the ring of integers $\Ha=\oink_E$ for most maximal
semisimple commutative subalgebras $E$ of $A=\matrici_2(B)$.
In fact, when  $A$ is a matrix algebra over $K$, of arbitrary dimension, the corresponding
subcomplex was described in \cite{scffgeo}.  Unfortunately, the technique use therein cannot be
applied to more general central simple algebras.  B. Linowitz and T. Shemanske computed, in a recent 
article, the higher dimensional analog of the branch, 
for the ring of integers of an unramified maximal semisimple commutatative subalgebra of top dimension,
without any restriction on $A$ \cite[Theorem 2.1]{lino5}. We will not go into details of the general definition
of Bruhat-Tits buildings (or BTBs) for $\matrici_n(B)$,
 since the present article deals exclusively with trees. The interested reader can find the details in \cite{build}
or \cite{build2}. Let us just recall that
the BTB is a simplicial complex containning some subcomplexes called appartments whose underlying 
topological space is the $(n-1)$-dimensional affine space.  In this language, their result can be rephrased as follows:
\begin{quote}
\textit{Let $B$ be an $m^2$-dimensional central division algebra over a local field $K$. Let 
$\Ha=\oink_L\cong\prod_{i=1}^r\oink_{L_i}$ be an order in $\mathbb{M}_n(B)$, where each factor
$\oink_{L_i}$ is the ring of integers of a field $L_i$ that is unramified over $K$, and 
$\sum_{i=1}^r[L_i:K]=mn$. Then $\mathfrak{S}(\Ha)$ is contained in an appartment and
its underlying topological space corresponds,with the appropiate identifications, to an $(r-1)$-dimensional linear subspace.}
\end{quote}
In particular, if $\Ha$ is the ring of integers in a maximal unramified subfield, then it is contained in a unique maximal order.
The techniques developped in this work yield the following partial generalization:
\begin{theorem}\label{tho}
Let $B$ be an $n^2$-dimensional central division algebra over a local field $K$. Let 
$\Ha=\oink_L$ be the ring of integers of a $2n$-dimensional commutative semisimple subalgebra
$L\subseteq\mathbb{M}_2(B)$. Then:
 \begin{enumerate}
\item If $L$ is a field and the ramification index $e(L/K)$ is odd, then $\oink_L$ is contained in a unique maximal order.
\item If $L$ is a field, containing a degree $n$ extension $F$ of $K$,  and if $e(L/F)=2$,
 then $\mathfrak{S}(\Ha)$ is a path of length $\frac{2n}{e(L/K)}$.
\item If $L$ is not a field then $\mathfrak{S}(\Ha)$ is a maximal path.
\end{enumerate}
\end{theorem}

If $L$ is a field and  $e(L/K)$ is even, the technical condition in the second case is satisfied,
 for example, whenever $L$ is Galois over the largest unramified subfield $L_{\mathrm{unr}}$,
 as we can choose $F$ as the invariant field $L^{\langle\sigma\rangle}$ for any involution 
$\sigma\in\mathrm{Gal}(L/L_{\mathrm{unr}})$. On the other hand,
 if $L/K$ is a quartic extension whose associated Galois group is $A_4$, there is no intermediate
quadratic extension, and the shape of the branch $\mathfrak{S}(\Ha)$ cannot be determined by the methods
presented here. For the existence of such extensions, see \cite{WJ07}.

The branch  $\mathfrak{S}(\Ha)$ has been described only for a limited number of order families in the
higher dimensional case.  Aside from the results for commutative orders cited above,
we can only mention Shemanske's description of the set of split orders in \cite{soacpib}, which we used in
\cite{scffgeo} to compute spinor class fields for intersections of two maximal orders in higher
dimensional central simple algebras. 
 Our main result in this paper is a  description of the graph $\mathfrak{S}=\mathfrak{S}(\Ha)$ 
whenever $\Ha=\oink_L$ is the ring of integers of an arbitrary subalgebra $L$ of $B$, identified with a 
subalgebra of $\matrici_2(B)$ via the map  $\lambda\mapsto\lbmatrix\lambda00\lambda$. The precise result,
namely Theorem \ref{tha}, is given in next section by its technicality. Here we present two 
consequences of it.
 Both are counterexamples to natural extension to higher dimensions
 of know results in the theory of quaternion algebras.

 The first one is a result due to F.-T. Tu \cite{tu}:
\begin{quote}
\textit{An order in $\matrici_2(K)$ is the intersection of a finite family of maximal orders if and only if it is
the intersection of three maximal orders.}
\end{quote}
This result fails to generalize to $\matrici_2(B)$ for a non-commutative algebra $B$.

\begin{theorem}\label{thb}
Let $B$ be a non-commutative $n^2$-dimensional division algebra over a local field $K$. Let 
$\Ha\subseteq\matrici_2(B)$ be an order contained
in $$\Da^{[2]}=\{x\in\Da|\bar{x}\in Z(\Da/\mathcal{M}^2_\Da)\},$$ where $\mathcal{M}_\Da$ is the unique maximal
 bilateral ideal of the maximal order $\Da$, $\bar{x}$ denotes the image on the quotient, while $Z(R)$
denotes the center of the ring $R$. Then $\Ha$ cannot be the intersection of
 three maximal orders in $\matrici_2(B)$. However, the order $\Da^{[2]}$
is the intersection of a finite family of maximal orders.
\end{theorem}

This application requires only the case $L=B$ in Theorem \ref{tha}.
Unfortunately, our method gives no hint on how to compute the minimal integer $t$, 
such that every finite intersection of maximal orders
is the intersection of $t$ maximal orders, nor to determine whether this number exists.

Our second result is an application of Theorem \ref{tha} to selectivity of commutative orders into genera
of non-maximal orders of maximal rank (full orders). This theory has been
studied by several authors in the quaternionic case  \cite{Guo}, 
\cite{Chan}, \cite{lino1}, specially because of the importance of the particular case of Eichler orders
in the spectral theory of hyperbolic varieties. 
For higher dimensional central simple algebras, selectivity of maximal integral domains into maximal orders
has been studied by Linowitz and Shemanske in \cite{lino3} and  \cite{lino4}.
 However, little is known about  selectivity  of commutative orders into
non-maximal full orders  in the higher dimensional case. For intersection of two maximal orders,
this has been done by the second author in some cases \cite{scffgeo}.

In order to make next result precise, we recall a few basic facts from the theory of genera of orders:
\begin{itemize}
\item A genus is a maximal set of full orders in a global central simple algebra whose 
completions at any archimedean place
are conjugates. A commutative order is called selective if it embeds into some, 
but not all, the orders in a particular genus.
\item In an algebra whose dimension exceeds $4$, or in an indefinite quaternion algebra, the proportion 
of conjugacy classes of orders containing a copy of a commutative order $\Ha$,  in any given genus
$\mathrm{gen}(\Da)$, is often of the form $[F:E]^{-1}$, where $E$ in the base field and 
$F=F(\Da|\Ha)$ is a class field called the representation field.
\item In principle, the representation field above can be undefined. Although some examples are known if we
remove the commutativity condition above, it might as well be the case that the representation field
is defined for any commutative order $\Ha$, and any full order $\Da$. This has only been proved
 with some restriction, for instance, it is known when $\Da$ is maximal \cite[Th. 1.1]{abelianos},
or when $\Da$ is strongly unramified in an algebra without partial ramification \cite[Prop. 5.1]{scffgeo}.
\end{itemize}
In particular, when the representation field is defined, the order $\Ha$ is
 selective if and only if the representation field contains $E$ properly.
We recall the precise definitions and part of the general theory in \S\ref{sec3} below.
See \cite{abelianos} or \cite{cyclic} for details.
In a quaternion algebra, we have next result  \cite[Theorem 1.2 and Theorem 1.3]{eichler2}:
\begin{quote}
\emph{If $\Da$ is a finite intersection of  maximal orders in a quaternion 
algebra, the representation field $F(\Da|\Ha)$ splits completely
 at any place where $\Ha$  is contained in infinite many 
local maximal orders.}
\end{quote}

For algebras of the form  $\matrici_2(B)$, where $B$ is a division algebra, we provide the following counter-example
to the straightforward generalization of the previous result:

\begin{theorem}\label{thc}
Let $E$ be a number field, and let $H$ be its Hilbert class field. Let $M$ be a quadratic extension of $E$
contained in $H$. Let $\wp$ and $\wp'$ be two finite places of $E$ that are inert for $M/E$ and let $\tilde B$
be a quaternion division $E$-algebra ramifying only at $\wp$ and $\wp'$. Then there exists a quadratic extension
$N$ of $M$, an embedding $\phi:N\rightarrow\matrici_2(\tilde B)$,
 a rank-4 $\oink_E$-order $\Ha\subseteq N$, and
 a rank-16 $\oink_E$-order $\Da\subseteq\matrici_2(\tilde B)$, satisfying the following conditions:
\begin{enumerate}
\item $M\subseteq F\big(\Da|\phi(\Ha)\big)$, in particular, neither $\wp$ nor $\wp'$ split completely in the representation field,
\item $\Da$ is a finite intersection of maximal orders, and
\item there exist infinitely many local maximal orders in $\matrici_2(\tilde B_\wp)$ containing $\phi(\Ha_\wp)$,
and the same holds at $\wp'$. 
\end{enumerate}
\end{theorem}
This is a new phenomenon that needs to be taken into account when studying selectivity in higher dimensions.
Note that the order $\Ha$ provided by the theorem is selective for the genus $\mathrm{gen}(\Da)$.
 We construct this counterexample by taking advantage of the periodicity
 of the branches described in Theorem \ref{tha}. We require only the case where
$L=M_\wp$ is an unramified quadratic extension of the local field $K_\wp$.

The subgraphs of the form $\mathfrak{S}=\mathfrak{S}(\Ha)$  in the BT-tree,
 which we call admisible branches in the sequel, are in correspondence with
the intersections of maximal orders. In fact, the maps $\Ha\mapsto\mathfrak{S}(\Ha)$,
and $\mathfrak{S}\mapsto\Ha_{\mathfrak{S}}=\bigcap_{\Da\in V(\mathfrak{S})}\Da$
are inverse bijections on these two restricted sets.
 In \cite{eichler2}, we gave a description of all admisible branches 
(and therefore of all intersections of maximal orders) when $B$ is a field. The finite
intersections, or equivalently the intersections of three maximal orders,
 had already been determined by F.-T. Tu in \cite{tu}. 
The set of  admissible branches is extremely simple in the field case. The full description of it is recalled
in Proposition \ref{p21} bellow. The examples presented in this work show that the situation is 
far more complex for non-commutative division algebras.

\section{Graphs, subdivisions and gluing}

Throughout, by a graph we mean a 5-tuple $\Gamma=(V,A,s,t,r)$ where:
\begin{itemize}
\item $V=V(\Gamma)$ and $A=A(\Gamma)$ are disjoint sets, called the vertex set and the edge set of $\Gamma$,
\item $s,t:A\rightarrow V$ are two maps call source and target, 
 \item $r:A\rightarrow A$ is called the reverse,
\end{itemize}
 and they are assumed to satisfy the following conditions, for any $a\in A$:
\begin{equation}\label{graphdef}
s(a)=t\big(r(a)\big),\qquad r\big(r(a)\big)=a,\qquad r(a)\neq a.
\end{equation}
The valency of a vertex $v$ is the number of edges $a$ satisfying $s(a)=v$ or equivalently, 
the number of edges $a$ satisfying $t(a)=v$. A node is a vertex of valency $3$ or larger. A simplicial map
$\gamma:\Gamma\rightarrow\Gamma'$, where $\Gamma'=(V',A',s',t',r')$ is a second graph, is a pair of maps
$\gamma_V:V\rightarrow V'$ and $\gamma_A:A\rightarrow A'$ satisfying the identities:
$$r'\circ\gamma_A= \gamma_A \circ r,\qquad s'\circ\gamma_A= \gamma_V \circ s,\qquad
t'\circ\gamma_A= \gamma_V \circ t.$$ This defines a category $\mathfrak{Graphs}$ whose objects are graphs
and whose morphisms are simplicial maps.  

 For any positive integer $n$, we set $\mathcal{C}_n=\{0, 1, \dots , n-1\}$ 
and  $\mathcal{C}_n'=\{1, \dots , n-1\}$.
If $\Gamma=(V,A,s,t,r)$ is a graph, we denote by $\Gamma/n$ the $n-$th subdivision of $\Gamma$, i.e.,
 the 5-tuple $(V_n,A_n,s_n,t_n,r_n)$, where each set or function is defined as follows:
\begin{enumerate}
\item $V_n=V \cup A_n'$ where $A_n'$ is the quotient of the cartesian product
$A \times \mathcal{C}_n' \times \{V\}$ by the equivalence relation $\sim$ whose only nontrivial
instance is $(a,k,V)\sim\big(r(a),n-k,V)$ for any edge $a$ and any $k\in\mathcal{C}_n'$. 
Note  the last $V$ in each element of $A_n'$ is simply a reminder that the elements of $A_n'$
 are vertices, not edges. The class of  $(a,k,V)$ is denoted $(a,k)_V$.
\item  $A_n=A \times \mathcal{C}_n \times \{A\}$,
and we write $(a,k)_A$ instead of $(a,k,A)$ for the sake of uniformity.
\item The new reverse $r_n:  A_n \longrightarrow A_n$ is defined by the formula
$r_n\left((e,k)_A\right)=\big(r(e),n-1-k\big)_A$.\item
 The new source and target are defined by the following formulas:
$$s_n\big((e,k)_A\big)=\begin{cases} s(e)& \textnormal{if } k=0 \\(e,k)_V &\textnormal{ otherwise} 
\end{cases},
\quad t_n\big((e,k)_A\big)=\begin{cases}  t(e) & \textnormal{ if } k=n-1\\(e,k+1)_V &\textnormal{otherwise} \end{cases}. $$
\end{enumerate}

To prove that $\Gamma/n$ thus defined is a graph, we just need to check, case by case,
 that the maps $r_n$, $s_n$, and $t_n$ satisfy all conditions in (\ref{graphdef}).
 As an example, we assume $k\neq 0$ and observe that \small
$$t_n\circ r_n\big((e,k)_A\big)=t_n\Big(\big(r(e),n-1-k\big)_A\Big)=\big(r(e),n-k\big)_V
= (e,k)_V= s_n\big((e,k)_A\big).$$ \normalsize
All other cases are analogous. In current literature,
a graph $\Gamma$ is said to be subdue to a graph $\Delta$ if there is a natural number $n$
such that the $n$-th subdivision  $\Gamma/n$ of $\Gamma$ is isomorphic, as an element of
$\mathfrak{Graphs}$, to a subgraph of $\Delta$. 
 The particular case $\Gamma/2$ is usually called the barycentric subdivision of $\Gamma$. 
Note that for any graph $\Gamma$, the nodes of $\Gamma/n$
can be identified with the nodes in $\Gamma$, while the vertices in $A_n'$ have always valency $2$. 

Let $\mathfrak{R}$ be 
the graph whose vertex set is $V(\mathfrak{R})=\enteri$ and there is
an edge $a_{i,j}$ such that $s(a_{i,j})=i$ and $t(a_{i,j})=j$ for every pair $(i,j)\in\enteri^2$
satisfying $|i-j|=1$. This graph is called the real line graph.
An integral interval is a connected subgraph $\mathfrak{I}\subseteq\mathfrak{R}$. For any subgraph $\mathfrak{G}\subseteq\mathfrak{T}$, a path in $\mathfrak{G}$ is an injective simplicial map
$\gamma:\mathfrak{I}\rightarrow \mathfrak{G}$, where $\mathfrak{I}$ is an integral interval.
If $\mathfrak{I}=\mathfrak{R}$, we call it a maximal path.
 Two path 
$\gamma:\mathfrak{I}\rightarrow \mathfrak{T}$ and 
$\gamma':\mathfrak{I}'\rightarrow \mathfrak{T}$ are equivalent if there exists an lnteger 
$n$ such that $\gamma=\gamma'\circ\sigma[n]$ where $\sigma[n]$ is the shift by $n$, i.e., the 
simplicial map defined on vertices by $\sigma[n]_V(i)=i+n$.
If, up to equivalence, $\mathfrak{I}=\mathfrak{N}$ is the interval whose vertex set is
 $V(\mathfrak{N})=\{0,1,2,\dots\}$, then
$\gamma$ is called a ray. A line is the image of a path.
The length of an interval with $n+1$ vertices is defined as $n$. This defines the length of a path or line
in an obvious way. For any vertex $\Da$ in $\mathfrak{T}$, and for every
natural number $r$, the ball of radius $r$, denoted $\mathfrak{B}(\Da;r)$,
 is the largest subgraph containing no vertex at a distance larger than $r$ from $\Da$.
It can also be defined as the union of all lines of length $r$ containing $\Da$.
 In this notations, next result follows from the proof of Theorem 1.4 in \cite{eichler2}:
\begin{proposition}\label{p21}
Assume $B=K$ is a field. Then
all admisible branches $\mathfrak{S}=\mathfrak{S}(\mathfrak{H})$ are described in the following list:
\begin{enumerate}
\item The $r$-thick line $\mathfrak{S}=\bigcup_{i\in V(\mathfrak{I})}\mathfrak{B}(\gamma(i);r)$, 
where $\gamma:\mathfrak{I}\rightarrow\mathfrak{T}$ is a path.
\item The infinite foliage $\mathfrak{S}=\bigcup_{i\in V(\mathfrak{N})} \mathfrak{B}(\nu(i);i)$, 
where $\nu:\mathfrak{N}\rightarrow \mathfrak{T}$ is a ray. This type of brach is attained only by the order
generated by a nilpotent element. 
\item $\mathfrak{S}=\mathfrak{T}$, attained by the trivial order $\Ha=\oink_K$.
\end{enumerate}
\end{proposition}

By a $(t,r)$-rose, as in Fig. 1, we mean a tree $\mathfrak{Z}$ with one distinguished vertex $c$
of valency $t$, called the center,  such that every path in $\mathfrak{Z}$
 starting from $c$ has length $r$ or less.  Attaching a rose to a 
vertex $v$ of a graph $\mathfrak{G}$ means identifying $c$ with $v$ in the disjoint union 
of the graph and the rose. If we let $\mathfrak{P}=(\{p\},\emptyset,\emptyset,\emptyset,\emptyset)$
be the graph with one vertex and no edges, this attachment can be seen as an amalgamated sum
in $\mathfrak{Graphs}$ corresponding to the maps $\alpha:\mathfrak{P}\rightarrow\mathfrak{Z}$ and
$\beta:\mathfrak{P}\rightarrow\mathfrak{G}$ satisfying $\alpha_V(p)=c$ and $\beta_V(p)=v$. 
The corresponding definition for families is analogous. Two $(t,r)$-roses are isomorphic if there is an
isomorphism in $\mathfrak{Graphs}$ between them preserving the center. We call $r$ the radius of the rose.

\begin{figure}
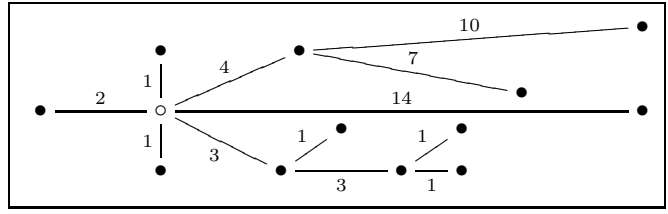

\[
\fbox{ \xygraph{
!{<0cm,0cm>;<.8cm,0cm>:<0cm,.8cm>::} 
!{(7,0) }*+{\circ}="O2"
!{(15,0) }*+{\bullet}="a2" !{(9.3,1) }*+{\bullet}="b2"  !{(15,1.4) }*+{\bullet}="b3"
!{(13,.3) }*+{\bullet}="b4" !{(7,1) }*+{\bullet}="c2"
 !{(5,0) }*+{\bullet}="e2" 
!{(7,-1) }*+{\bullet}="g2" !{(9,-1) }*+{\bullet}="h2" !{(11,-1) }*+{\bullet}="h3" !{(12,-1) }*+{\bullet}="h4" 
!{(12,-0.3) }*+{\bullet}="h5" !{(10,-0.3) }*+{\bullet}="h6" 
"O2"-^{14}"a2" "O2"-^{4}"b2" "O2"-^{1}"c2"
"O2"-_{2}"e2"
"O2"-_{1}"g2"  "O2"-_{3}"h2" "h2"-_{3}"h3" "h3"-_{1}"h4"  "h3"-^{1}"h5"
"h2"-^{1}"h6"  "b2"-^{10}"b3" "b2"-^{7}"b4" 
}} 
\]
\caption{A  $(6,14)$-rose. The white circle denotes the center.}
\end{figure}

\begin{theorem}\label{tha}
Let $B$ be a $n^2$-dimensional central division algebra over a local field $K$.
Assume that the residue field of $K$ is the field $\mathbb{F}_q$ with $q$
elements.
 Let $L\subseteq B$ be a simple subalgebra. Let $C\subseteq B$ be the centralizer of $L$.
Let $e=e(L/K)$ be the ramification index, and let $e'=e(B/L)=n/e$. Let
$f=f(L/K)$ be the residual degree.
Then $\mathfrak{S}(\Ha)$  is isomorphic to the graph obtained by attaching an isomorphic 
$(q^n-q^{e'},e')$-rose to every node of $\mathfrak{T}(C)/f$.
\end{theorem}

 If $e'=1$, there is a unique possible rose, 
so the shape of the graph is completely determined by Theorem 1.1.
This is the case when $L$ is a maximal, totally ramified subfield, or when $L=B$.
 On the other hand, when $e'=n$, no roses need to be attached, and the graph is determined by  Theorem 1.1 also.
This holds for an unramified subfield.
 Intermediate cases are more involved, but we can still give a full description of the branch 
$\mathfrak{S}(\Ha)$ in this case. We pospone this until \S\ref{sec4}, since this general
case is not used in the aplications. See Figure 2 for some examples.

\newcommand\clubs{\textnormal{\tiny$\clubsuit$}}
\begin{figure}
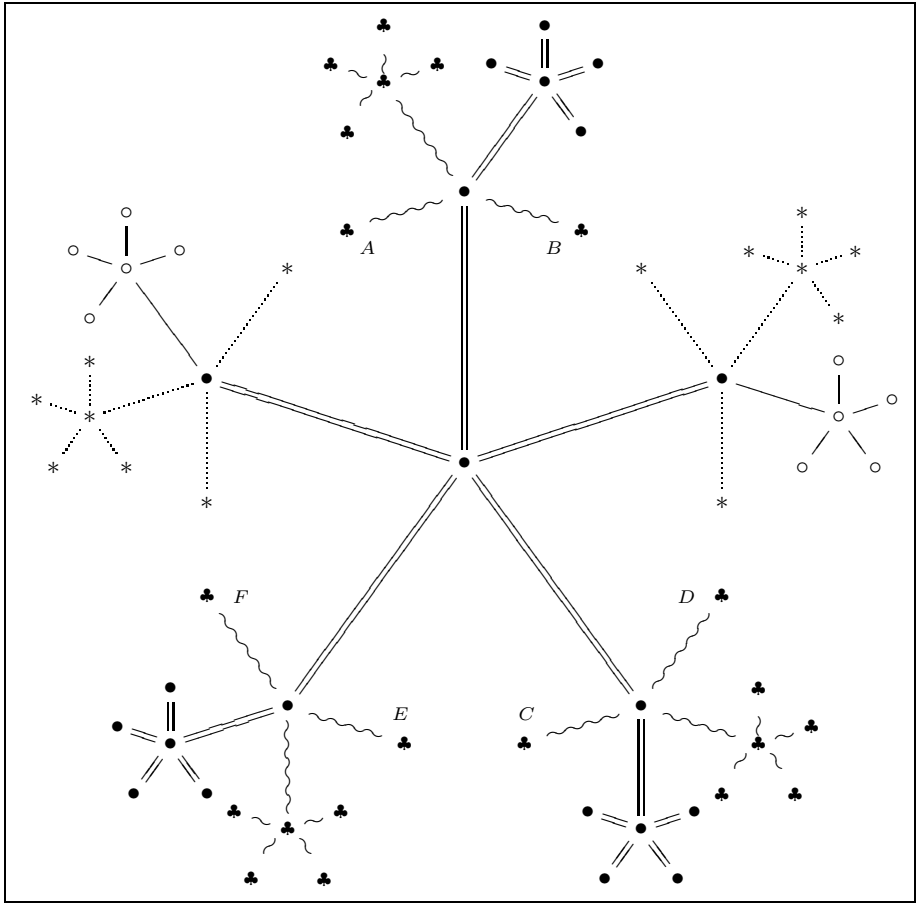

\[
\fbox{ \xygraph{
!{<0cm,0cm>;<.4cm,0cm>:<0cm,.4cm>::} 
!{(-3.2,7) }*+{{}^A}="mna" !{(3,7) }*+{{}^B}="mnb"
!{(2.1,-8.5) }*+{{}^C}="mnc" !{(7.4,-4.6) }*+{{}^D}="mnd"
!{(-2.1,-8.5) }*+{{}^E}="mne" !{(-7.4,-4.6) }*+{{}^F}="mnf"
!{(0,0) }*+{\bullet}="O"
!{(0,9) }*+{\bullet}="a1" !{(8.559,2.781) }*+{\bullet}="a2"  
 !{(5.877,-8.09) }*+{\bullet}="a3"    !{(-5.877,-8.09) }*+{\bullet}="a4"
!{(-8.559,2.781) }*+{\bullet}="a5"
!{(-3.89,7.736) }*+{\clubs}="b12"  
 !{(-2.671,12.677) }*+{\clubs}="b13"  !{(2.671,12.677) }*+{\bullet}="b14"
 !{(3.89,7.736) }*+{\clubs}="b15"
!{(8.559,-1.31) }*+{*}="b21"  !{(5.888,6.458) }*+{*}="b23"  
 !{(11.23,6.458) }*+{*}="b24" !{(12.449,1.517) }*+{\circ}="b25"
 !{(5.877,-12.181) }*+{\bullet}="b31"  !{(1.987,-9.354) }*+{\clubs}="b32"
 !{(8.548,-4.413) }*+{\clubs}="b34"  !{(9.767,-9.354) }*+{\clubs}="b35"
!{(-5.877,-12.181) }*+{\clubs}="b41" 
!{(-9.767,-9.354) }*+{\bullet}="b42"  !{(-8.548,-4.413) }*+{\clubs}="b43"
 !{(-1.987,-9.354) }*+{\clubs}="b45"  
!{(-8.559,-1.31) }*+{*}="b51" 
!{(-12.449,1.517) }*+{*}="b52"  !{(-11.23,6.458) }*+{\circ}="b53"
 !{(-5.888,6.458) }*+{*}="b54"  
!{(-12.449,3,376) }*+{*}="c521"  
 !{(-11.235,-0.155) }*+{*}="c523"  !{(-13.664,-0.155) }*+{*}="c524"
 !{(-14.218,2.091) }*+{*}="c525"
!{(12.449,3,376) }*+{\circ}="c251"  
 !{(11.235,-0.155) }*+{\circ}="c254"  !{(13.664,-0.155) }*+{\circ}="c253"
 !{(14.218,2.091) }*+{\circ}="c252"
!{(-9.767,-7.495) }*+{\bullet}="c421"  
 !{(-8.553,-11.026) }*+{\bullet}="c423"  !{(-10.982,-11.026) }*+{\bullet}="c424"
 !{(-11.536,-8.780) }*+{\bullet}="c425"
!{(9.767,-7.495) }*+{\clubs}="c351"  
 !{(8.553,-11.026) }*+{\clubs}="c354"  !{(10.982,-11.026) }*+{\clubs}="c353"
 !{(11.536,-8.780) }*+{\clubs}="c352"
!{(4.108,-11.607) }*+{\bullet}="c315"  
 !{(4.663,-13.853) }*+{\bullet}="c314"  !{(7.092,-13.833) }*+{\bullet}="c313"
 !{(7.646,-11.607) }*+{\bullet}="c312"
!{(-4.108,-11.607) }*+{\clubs}="c412"  
 !{(-4.663,-13.853) }*+{\clubs}="c413"  !{(-7.092,-13.833) }*+{\clubs}="c414"
 !{(-7.646,-11.607) }*+{\clubs}="c415"
!{(-2.671,14.537) }*+{\clubs}="c131"  
 !{(-0.903,13.252) }*+{\clubs}="c132"  !{(-3.886,11.006) }*+{\clubs}="c134"
 !{(-4.44,13.252) }*+{\clubs}="c135"
!{(2.671,14.537) }*+{\bullet}="c141"  
 !{(4.44,13.252) }*+{\bullet}="c142"  !{(3.886,11.006) }*+{\bullet}="c143"
 !{(0.903,13.252) }*+{\bullet}="c145"
!{(11.23,8.318) }*+{*}="c241"  
 !{(12.999,7.033) }*+{*}="c242"  !{(12.445,4.787) }*+{*}="c243"
 !{(9.462,7.033) }*+{*}="c245"
!{(-11.23,8.318) }*+{\circ}="c531"  
 !{(-9.462,7.033) }*+{\circ}="c532"  !{(-12.445,4.787) }*+{\circ}="c534"
 !{(-12.999,7.033) }*+{\circ}="c535"
"O"-@{=}"a1" "O"-@{=}"a2" "O"-@{=}"a3" "O"-@{=}"a4" "O"-@{=}"a5"
 "a1"-@{~}"b12" "a1"-@{~}"b13" "a1"-@{=}"b14" "a1"-@{~}"b15"
"a2"-@{.}"b21"  "a2"-@{.}"b23" "a2"-@{.}"b24" "a2"-"b25"
"a3"-@{=}"b31"  "a3"-@{~}"b32" "a3"-@{~}"b34" "a3"-@{~}"b35"
"a4"-@{~}"b41"  "a4"-@{=}"b42" "a4"-@{~}"b43" "a4"-@{~}"b45"
"a5"-@{.}"b51"  "a5"-@{.}"b52" "a5"-"b53" "a5"-@{.}"b54"
"b52"-@{.}"c521"  "b52"-@{.}"c523"  "b52"-@{.}"c524"  "b52"-@{.}"c525"
"b25"-"c251"  "b25"-"c253"  "b25"-"c254"  "b25"-"c252"
"b35"-@{~}"c351"  "b35"-@{~}"c353"  "b35"-@{~}"c354"  "b35"-@{~}"c352"
"b14"-@{=}"c141"  "b14"-@{=}"c142"  "b14"-@{=}"c143"  "b14"-@{=}"c145"  
"b42"-@{=}"c421"  "b42"-@{=}"c423"  "b42"-@{=}"c424"  "b42"-@{=}"c425"  
"b31"-@{=}"c312"  "b31"-@{=}"c313"  "b31"-@{=}"c314"  "b31"-@{=}"c315"  
"b13"-@{~}"c131"  "b13"-@{~}"c132"  "b13"-@{~}"c134"  "b13"-@{~}"c135"
"b41"-@{~}"c413"  "b41"-@{~}"c412"  "b41"-@{~}"c414"  "b41"-@{~}"c415"  
"b53"-"c531"  "b53"-"c532"  "b53"-"c534"  "b53"-"c535"  
"b24"-@{.}"c241"  "b24"-@{.}"c243"  "b24"-@{.}"c242"  "b24"-@{.}"c245"
}} 
\]
\caption{ Bullets and double lines denote part of the branch $\mathfrak{S}(\oink_B)\subseteq\mathfrak{T}(B)$ 
when $B$ is the quaternion division algebra over $\mathbb{Q}_2$. Here $C=Z(B)\cong\mathbb{Q}_2$, 
while $e=f=n$. We show $\mathfrak{S}(\oink_B)$ as the intersection of $\mathfrak{S}(\oink_L)$, where
$L/K$ is an unramified
 quadratic extension (denoted by white circles and single lines), and the branch $\mathfrak{S}(\pi_B)$,
of a uniformizing parameter (denoted by clubs and wiggly lines). The clubs $A,\dots,F$ have
valency one in the latter branch (c.f. \S3).}
\end{figure}

In all that follows, by an end $e$ of a graph $\mathfrak{G}$ we mean an equivalence class of rays, 
where two rays $\gamma,\gamma':\mathfrak{N}\rightarrow \mathfrak{G}$ are equivalent if there
is a fix integer $t$ such that $\gamma'(n)=\gamma(n+t)$ for every sufficiently large positive
integer $n$. Ends of sub-graphs are naturally identified with ends of the larger graph,
and notations like $e\in \mathfrak{G}$ should be understand in this sense,
 if $e$ is an end and $\mathfrak{G}$
is a subgraph of the Bruhat-Tits tree. We also use $\mathfrak{S}(a)$ instead of
$\mathfrak{S}\big(\oink_K[a]\big)$ for the branch of the order generated by a single
element $a\in B$.

\section{Local geometry of branches}

Let $B$, $K$, and $\oink_K$ be as in \S1, and let $\oink_B$ be the ring of integers of $B$. Let $\mathfrak{m}_B=\pi_B\oink_B$ be the unique maximal bilateral ideal of  $\oink_B$,
 and let $\mathfrak{m}_K=\pi_K\oink_K=
\mathfrak{m}_B\cap\oink_K$ be the maximal ideal of $\oink_K$.
 Let $\mathbb{B}=\oink_B/\mathfrak{m}_B\cong\mathbb{F}_{q^n}$ be the residue field of $B$, while
$\mathbb{K}=\oink_K/\mathfrak{m}_K\cong\mathbb{F}_{q}$ is the residue field of $K$.
 Note that $q$ is a prime power.
If $\rho$ denotes the absolute value in $B$, then $\oink_B=\{x\in B | \rho(x)\leq1\}$ and 
$\mathfrak{m}_B=\{x\in B | \rho(x)<1\}$.
 Recall that the column 
space $B^2=\textnormal{\scriptsize$\left(\begin{array}{c}B\\ B\end{array}\right)$\normalsize}$
 has a natural bi-module structure with matrices acting on it by left multiplication while scalars act on the right. This is the module structure considered throughout this paper. In particular, by an $\oink_B$-lattice in the
space $B^2$
we mean a lattice that is closed under right-multiplication by scalars in $\oink_B$ but not necessarily under left multiplication by scalar matrices. Let $\Da_0\subseteq \matrici_2(B)$ be a maximal order.
 Recall that 
$\Da_0=\mathrm{End}_{\oink_B}(\Lambda_0)$ for some $\oink_B$-lattice $\Lambda_0$. By a change of basis,
 we can assume that $\Lambda_0=\textnormal{\scriptsize$\left(\begin{array}{c}\oink_B\\ \oink_B\end{array}\right)$\normalsize}=\oink_B^2$, so that
$\Da_0=\matrici_2(\oink_B)=\lbmatrix {\oink_B}{\oink_B}{\oink_B}{\oink_B}$. Then any neighbor of
$\Da_0$ is a maximal order of the form $$\Da_{\stackrel{\rightarrow}v}=
\mathrm{End}_{\oink_B}[\Lambda_0\pi_B +\stackrel{\rightarrow}v\oink_B],\qquad \stackrel{\rightarrow}v\in
 \Lambda_0\backslash \Lambda_0\pi_B.$$
Note that $\oink_B\pi_B=\pi_B\oink_B$, whence $\Lambda_0\pi_B$ is indeed an $\oink_B$-lattice. Furthermore
$\Da_{\stackrel{\rightarrow}v}$ depends only on the one dimensional vector space spanned
by the image of $\stackrel{\rightarrow}v$ on the $\mathbb{B}$-vector space
$\Lambda_0/ \Lambda_0\pi_B\cong\mathbb{B}^2$. This space determines a unique element 
$a$ in the projective line $\mathbb{P}^1(\mathbb{B})$. For an arbitrary order $\Ha$, the vertex $\Da_0$ belongs to $S(\Ha)$, by definition, if and only if $\Ha\subseteq\Da_0$. If we assume $\Da_0$ is as above, the neighbors of
$\Da_0$ in the branch $\mathfrak{S}(\Ha)$ correspond to the  one dimensional $\mathbb{H}$-submodules of
$\mathbb{B}^2$, where 
$$\mathbb{H}=(\Ha+\pi_B\Da_0)/\pi_B\Da_0\subseteq\Da_0/\pi_B\Da_0=\matrici_2(\mathbb{B}),$$
 and therefore its number is $0$, $1$, $2$, or $q^n+1$.
 We classify the vertices in $\mathfrak{S}(\Ha)$ in isolated points, leaves,
bridges, and nodes, according to their respective number of neighbors,  as in Fig. 3. Note that the valency of
a node is always $q^n+1$.

\begin{figure}
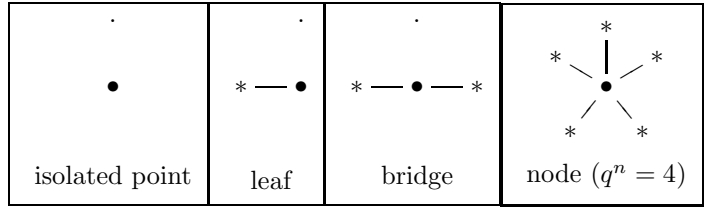

\[
\fbox{ \xygraph{
!{<0cm,0cm>;<.8cm,0cm>:<0cm,.8cm>::} !{(1,1) }*+{\bullet}="a" !{(1,2.1) }*+{.}="f"
!{(1,-.5)}*+{\textnormal{isolated point}}="e"   }} 
\fbox{ \xygraph{
!{<0cm,0cm>;<.8cm,0cm>:<0cm,.8cm>::} !{(1,1) }*+{\bullet}="a" !{(1,2.1) }*+{.}="f"
!{(0,1)}*+{*}="b"
!{(.5,-.54)}*+{\textnormal{leaf}}="e"  "b"-"a" }} 
\fbox{ \xygraph{
!{<0cm,0cm>;<.8cm,0cm>:<0cm,.8cm>::} !{(1,1) }*+{\bullet}="a" !{(1,2.1) }*+{.}="f"
!{(0,1)}*+{*}="b" !{(2,1)}*+{*}="c" 
!{(1,-.5)}*+{\textnormal{bridge}}="e"  "b"-"a" "a"-"c" }} 
 \fbox{ \xygraph{
!{<0cm,0cm>;<.8cm,0cm>:<0cm,.8cm>::} !{(1,1) }*+{\bullet}="a" 
!{(1.6,0.25)}*+{*}="b" !{(0.16,1.5)}*+{*}="c" !{(1.84,1.5) }*+{*}="d" !{(1,2) }*+{*}="f" !{(0.4,0.25) }*+{*}="g"
!{(1,-.48)}*+{\textnormal{node ($q^n=4$)}}="e"  "b"-"a" "a"-"c" "a"-"d" "a"-"f" "a"-"g"}} 
\]
\caption{Types of vertices in a branch.}
\end{figure}

We know that the branch $\mathfrak{S}(\Ha)$ is connected since, 
for any two maximal orders, $\Da_1$ and $\Da_2$,
 there is a choice of basis in which they have the form 
\begin{equation}\label{ddd}
\Da_1=\bbmatrix {\oink_B}{\oink_B}{\oink_B}{\oink_B},\qquad \Da_2=
\bbmatrix {\oink_B}{\pi_B^d\oink_B}{\pi_B^{-d}\oink_B}{\oink_B},
\end{equation}
for a unique possitive integer $d$, the distance between the orders,
 and therefore their intersection is contained in every maximal order in the path joining them.
We conclude that isolated points can only be present if $\Ha$ is contained in a unique maximal order.
This is the case if and only if the natural inclusion $\mathbb{H}\hookrightarrow\matrici_2(\mathbb{B})$ is an irreducible
representation.

The BT-tree is canonically isomorphic to the ball-tree of $B$, i.e., the tree whose vertices corresponds to
the closed balls on $B$, and where two balls are neighbors if and only if one is a maximal proper sub-ball of the other. The
maximal order corresponding to the ball $U$ with center $a$ and radius $\rho(c)$, for $a,c\in B$, is the endomorphism ring
$\Da_U$ of the lattice generated by the column vectors $\left(\begin{array}{c}c\\0\end{array}\right)$ and $\left(\begin{array}{c}a\\1\end{array}\right)$. 
Note That if $U=U_0$ is the radius $1$ ball centered at $0$, then $\Da_U=\matrici_2(\oink_B)$. 
Under this correspondence, we have
\begin{equation}\label{mob}
\Da_{U'}=\left(\begin{array}{cc}x&y\\0&z\end{array}\right)\Da_U\left(\begin{array}{cc}x&y\\0&z\end{array}\right)^{-1},\textnormal{ where }U'=(xU+y)z^{-1}.
\end{equation}
In what follows we denote by $\rho(U)=\rho(c)$ the radius of any ball $U=cU_0+a$.
We also recall that the set of ends in the BT-tree is in correspondence with
the projective line $\mathbb{P}^1(B)=B\cup\{\infty\}$. This identification is compatible
with the identification bwtween the Bruhat-Tits tree and the ball tree, in the sense that one representative of
the end corresponding to an element $b\in B$ is a ray whose
vertices are balls of decreasing radii having intersection $\{b\}$,
while $\infty$ is the end corresponding to any ray whose vertices are balls of increasing radii.  
The proof of all these facts is entirely analogous to the one given in \cite{AAC} for the field case.

As usual, for any division algebra $L\subseteq B$ 
we define the residual degree $f(B/L)$ as the degree 
$[\mathbb{B}:\mathbb{L}]$ of the corresponding extension 
of residual fields, and the ramification degree by $e(B/L)=
\log_{\rho(\pi_B)}\rho(\pi_L)$, where $\pi_B$ and $\pi_L$ are the corresponding uniformizing parameters.
Next result is fundamental in our proof of Theorem \ref{tha}:

\begin{lemma}\label{l31}
Let $B$ be a central division algebra over the local field $K$.
If $L$ is a $K$-subalgebra of $B$ and $C$ is the centralizer of $L$ in $B$, then the following identities hold:
$$ e(L/K)=f(B/C),\qquad f(L/K)=e(B/C).$$ 
\end{lemma}

\begin{proof}
Recall that the residual Galois group $\mathrm{Gal}(\mathbb{B}/\mathbb{K})$ is generated by an element
$\sigma$ satisfying $\sigma(\bar{b})=\overline{\pi_Bb\pi_B^{-1}}$ for any $b\in\oink_B$.
For every element $\lambda\in \oink_L$, its image $\bar{\lambda}$ in the residue field 
$\mathbb{L}=\oink_L/\pi_L\oink_L$ satisfies $\sigma^{e(B/C)}(\bar{\lambda})=\overline{\pi_C\lambda\pi_C^{-1}}$. 
Since $\pi_C$ and $\lambda$ commute,
we conclude that $f(L/K)$ divides $e(B/C)$. 
Since $L$ is the centralizer of $C$, we conclude that $f(C/K)$ divides $e(B/L)$, and therefore 
 $e(L/K)$ divides $f(B/C)$. Equality follows since
$$f(L/K)e(L/K)=\dim_KL=\frac{ \dim_KB}{\dim_KC}=e(B/C)f(B/C).$$ 
\end{proof}

In what follows, by a unit in $\matrici_2(B)$ we mean an element $\lambda\in \matrici_2(B)$ satisfying 
 an equation of the  form $x^n+a_{n-1}x^{n-1}+\dots+a_1x+a_0=0$ where each $a_i\in\oink_K$ is 
integral, while $a_0\in\oink_K^*$ is a unit. Note that a unit  is invertible in any $\oink_K$-order containing 
it. An element $\lambda\in B$ is a unit if and only if $\rho(\lambda)=1$.
If $\lambda\in\oink_B$ is not a unit, then $\rho(\lambda)<1$, so that  $1+\lambda$ is a unit.
 The same applies to any other division algebra. 
Certainly, for any order $\Da$, $\lambda\in\Da$ if and only if $1+\lambda\in\Da$,  
so by replacing $\lambda$ by $1+\lambda$, many computations can be reduced to units.

\begin{lemma}
Let $B$ be a central division algebra over the local field $K$.
Let $\Da$ be a maximal order in $\matrici_2(B)$. Then any unit in 
$\matrici_2(B)$ normalizing $\Da$ belongs to $\Da^*$. 
In particular, if $\lambda$ is a unit, then $\lambda\in\Da$ is equivalent to 
$\lambda\Da\lambda^{-1}=\Da$.

\end{lemma}

\begin{proof}
Replacing by conjugates if needed, we can assume that $\Da=\matrici_2(\oink_B)$. Then any element
$\lambda\in \matrici_2(B)$ satisfying $\lambda\Da \lambda^{-1}=\Da$ must also satisfy 
$\lambda\left(\begin{array}{c}\oink_B\\ \oink_B\end{array}\right)=
\left(\begin{array}{c}\oink_B\\ \oink_B\end{array}\right)u$ for some scalar $u\in B$.
Now let $\mu$ be the Haar measure on $k\times k$, and let $n:\matrici_2(B)\rightarrow K$ be
the reduced norm. It is well known\footnote{See for example Corollary 3 to Theorem 3 in \cite[\S1.2]{weil}.}
that $\mu(\lambda A)=\rho\Big(n(\lambda)\Big)^t\mu(A)$, for a fixed normalizing power $t$,
and any measurable set $A$. By applying the same result to $u\in B$, we must conclude that,
if $\lambda$ is a unit, then $u$ is a unit, and it can be assumed to be $1$. The result follows. 
\end{proof}

 Next result follows easily from the identity $f(L/K)=e(B/C)$, and the structure of the ball tree:

\begin{lemma}\label{l33}
Set $\Da_0=\matrici_2(\oink_B)$, where $B$ is a division $K$-algebra, and let $C$ be the centralizer of 
a sub-algebra $L\subseteq B$ as before. Let $\hat{\mathfrak{T}}(C)$
 be the minimal connected subgraph of $\mathfrak{T}(B)$ containing every end in $\mathbb{P}^1(C)\subseteq\mathbb{P}^1(B)$.
 then $\hat{\mathfrak{T}}(C)$ is isomorphic to $\mathfrak{T}(C)/f$,
where $f=f(L/K)$.
\end{lemma}

\paragraph{Proof of  Theorem \ref{tha}.}

Set $\Da_0=\matrici_2(\oink_B)$, as above, and identify $B$ with the set of scalar matrices $\lbmatrix b00b$. Note that such matrices act by conjugation on the ball tree as in (\ref{mob}).  
  Since the elements of $\oink_L^*$ act trivially on $\mathbb{P}^1(C)$
by conjugation, the branch $\mathfrak{S}(\oink_L)$ contains the vertex $\Da_U$ for every  ball $U$ with
a center $c\in C$. We conclude that $\hat{\mathfrak{T}}(C)$, as defined above,
 is contained in $\mathfrak{S}(\oink_L)$.  Now let
$\Da$ be a bridge in $\hat{\mathfrak{T}}(C)$. We claim that $\Da$ is also a bridge in $\mathfrak{S}(\oink_L)$.
Note that $\matrici_2(C)$ is the centralizer of $\oink_L$, and it acts transitively
on pairs of neighboring vertices of $\mathfrak{T}(C)$, i.e., pairs of nodes of $\hat{\mathfrak{T}}(C)$ at distance $f$. 
We conclude that,  conjugating by an element of $\matrici_2(C)$  if needed, we can
 assume $\Da=\Da_{U_r}$ for some ball $U_r$  centered at $0$ and radius
$\rho(\pi_B)^r$, where $1>\rho(\pi_B)^r>\rho(\pi_C)$. This can happen only if $\pi_C$ is not a uniformizer in $B$,
which is precisely the condition for $f>1$, so we actually have bridges.
It suffices to prove that $\oink_L$ contains an element $\lambda$ whose image $\bar{\lambda}\in \Da/\pi_B\Da\cong
\matrici_2(\mathbb{B})$ has different eigenvalues. Note that, if $U_0$ is the ball of radius $1$ centered at 
$0$, we have $\Da=\Da_{U_r}=a\Da_{U_0}a^{-1}$, where $a=\lbmatrix {\pi_B^r}001$, whence
the eigenvalues of the image $\bar{\lambda}\in\Da/\pi_B\Da$ of an element $\lambda\in\oink_L$ are
$l$ and $\sigma^r(l)$, where $l$ is the image of $\lambda$ in the residue field $\mathbb{L}$ of $L$, and therefore
they can be different as soon as $r$ is not divisible by $f(L/K)=e(B/C)$, but this is a consequence of the
inequalities $1>\rho(\pi_B)^r>\rho(\pi_C)$.

Now we prove that $\mathfrak{S}(\oink_L)$  is obtained by attaching a finite rose
to each node of $\hat{\mathfrak{T}}(C)$. In fact, since $\mathfrak{S}(\oink_L)$ is a tree, for every vertex $\Da\in\mathfrak{S}(\oink_L)$  there must exist a
minimal path joining $\Da$ to a vertex $v$ in $\hat{\mathfrak{T}}(C)$, 
which is necessarily a node from the preceding argument. We can assume it to be $v=\Da_{U_0}$
by the transitivity of the action of $\matrici_2(C)$ on the nodes of $\hat{\mathfrak{T}}(C)$. It suffices to prove
that the number of vertices in $\mathfrak{S}(\oink_L)$, for which this "closest $\hat{\mathfrak{T}}(C)$-node" is $\Da_{U_0}$, is
finite. If this were not the case, there would be an infinite sequence of balls 
$U_1,U_2,\dots$ of decreasing radii, corresponding to such
vertices, and a sequence of elements $b_i\in U_i$ converging to an element $b\in\oink_B$.
Enlarging the balls, if needed, 
we can assume they all contain $b$.
To fix ideas, set  $U_i=b+c_iU$.
This implies that the lattice generated by each pair of column vectors \footnotesize
$\left\{\left(\begin{array}{c}c_i\\0\end{array}\right),\left(\begin{array}{c}b\\1\end{array}\right)\right\}$
\normalsize, for arbitrarily small values of  $\rho(c_i)$, is invariant by any
scalar matrix in $\oink_L$. In particular, the lattice generated by the column vector 
\footnotesize $\left(\begin{array}{c}b\\1\end{array}\right)$
\normalsize is invariant,  but this can only be the case for  $b\in C$
as shown by the computation
$$\bbmatrix a00a\left(\begin{array}{c}b\\1\end{array}\right)=
\left(\begin{array}{c}aba^{-1}\\1\end{array}\right)a.$$

It remains to prove the bound $e'=e(B/L)$ on the radius of the rose. As before,
 we take $\Da\in\mathfrak{S}(\oink_L)$, and assume that $\Da$ is in the rose attached to 
$\Da_{U_0}$, whence $\Da=\Da_U$ for a ball
$U\subseteq U_0$ not contained in $c+U_1$ for any $c\in C$. 
Let $\epsilon$ be a center of this ball. Then we have
$$(1+\pi_L)\epsilon(1+\pi_L)^{-1}\equiv  \epsilon+\pi_L\big(\epsilon-\sigma^{e(B/L)}(\epsilon)\big) 
\quad (\mathop{\mathrm{mod}}\pi_L^2),$$
where $\sigma$ is as in the proof of Lemma \ref{l31}.
By Lemma \ref{l33}, we have $e(B/L)=f(C/K)=[\mathbb{C}:\mathbb{K}]$, where $\mathbb{C}$ is the residue field of $C$.
Hence, the condition that the residue
$\bar{\epsilon}$ is not in $\mathbb{C}$ implies
$$\rho\big((1+\pi_L)\epsilon(1+\pi_L)^{-1}-\epsilon\big)=\rho(\pi_L).$$ 
It follows that the radius of the ball $U$ cannot be smaller than
$\rho(\pi_L)$.
\qed

\section{Proof of Theorem \ref{thb} and Theorem \ref{thc}.}\label{sec3}

\begin{lemma}\label{blemma}
If $\mathfrak{B}$ is the ball of radius t in $\mathfrak{T}(B)$ centered at a maximal order $\Da$, then 
$\mathfrak{B}$ is admissible and the corresponding intersection of maximal orders is 
$$\Ha_{\mathfrak{B}}=\Da^{[t]}:=\{x\in \Da|\bar{x}\in Z(\Da/\mathcal{M}_\Da^t)\},$$ 
where $\mathcal{M}_\Da$ is the maximal bilateral ideal of $\Da$, and $\bar{x}$ denotes the image of 
$x$ in the quotient
$\Da/\mathcal{M}_\Da^t$.
\end{lemma}

\begin{proof}
We can assume $\Da=\matrici_2(\oink_B)$.
It is apparent from (\ref{ddd}) that every maximal
order whose distance to $\Da$ does not exceed $t$ contains  $\Da^{[t]}$, so
$\Ha_{\mathfrak{B}}\supseteq\Da^{[t]}$. On
the other hand, if $h$ belongs to every maximal order whose distance to $\Da$ does not exceed $t$, then 
$h=\lbmatrix ab{\pi_B^tc}d$ with $a,b,c,d\in\oink_B$, and its conjugate $uhu^{-1}$ has the same form
 for every unit $u\in\Da^*$.
Setting $u=\lbmatrix 0110$ we prove that $\pi_B^t$ divides $b$, while setting $u=\lbmatrix 10r1$
 we prove that $\pi_B^t$ divides $ra-dr$, for every element $r\in\oink_B$.
 Setting $r=1$ we obtain that $\bar{a}=\bar{d}$,
while other values of $r$ show that this element is central.  We conclude the contention
 $\Ha_{\mathfrak{B}}\subseteq\Da^{[t]}$,
and hence the equality. In particular, this implies $\mathfrak{S}(\Da^{[t]})=
\mathfrak{S}(\Ha_{\mathfrak{B}})\supseteq\mathfrak{B}$.
 On the other hand, if $\Da_2$ lies at a distance $d>t$ from $\Da_1=\Da$,
we can assume that $\Da_2$ is as in (\ref{ddd}), which fails to contain the element $h=\lbmatrix 10{\pi_B^t}1\in\Da^{[t]}$.
It follows that $\mathfrak{B}\supseteq \mathfrak{S}(\Da^{[t]})$, whence we obtain the equality,
and therefore $\mathfrak{B}$ is admisible as claimed.
\end{proof}

Note that the tree $\mathfrak{S}(\oink_B)$ contain no ball of radius $2$ (c.f. Figure 2),
 so given the preceeding lemma, 
Theorem \ref{thb} is a consequence of the following result:

\begin{lemma}\label{l42}
The intersection of any three maximal orders in $\matrici_2(B)$ contains a conjugate of $\oink_B$.
\end{lemma}

\begin{proof}
Given any three maximal orders, we can choose any three ends of the tree beyond them, as in Figure 4.
We claim that $\mathrm{PGL}_2(B)$ acts transitively on triples of diferent ends.   We can, 
therefore, assume that the ends beyond our three orders correspond to $1$, $0$, and 
$\infty$, whence $\oink_B$ is contained in the three given maximal orders by Theorem 1. It remains to prove the claim.
The result follows if we prove that every triple is in the orbit of $(\infty,0,1)$. The group of translations
$x\mapsto x+b$, a particular case of the transformation in (\ref{mob}),
 acts transitively on $B=\mathbb{P}^1(B)-\{\infty\}$. On the other hand
$$\lbmatrix 0110 \Da_{U_n}{\lbmatrix 0110}^{-1}=\Da_{U_{-n}},\qquad
\Da_{U_k}=\lbmatrix {\oink_K}{\pi_B^k\oink_K}{\pi_B^{-k}\oink_K}{\oink_K},$$
where $U_k$ denotes the ball of radius $\rho(\pi_B^k)$ centered at $0$,
whence conjugation by $\lbmatrix 0110$
allows us to replace $\infty$, i.e., the end corresponding to the ascending ray 
$\gamma:\mathfrak{N}\rightarrow\mathfrak{T}$ satisfying $\gamma_V(n)=\Da_{U_{-n}}$,
with $0$,  the end corresponding to the  ray $\gamma'$ satisfying $\gamma'_V(n)=\Da_{U_n}$.
 We conclude that
$\mathrm{PGL}_2(B)$ acts transitively on $\mathbb{P}^1(B)$. Then a map of the form
$x\mapsto x+b$ allows us to replace any finite end by $0$ while stabilizing $\infty$. Finally, a map of the form
$x\mapsto ax$ is transitive in $\mathbb{P}^1(B)-\{0,\infty\}$, but leaves $0$ and $\infty$ invariant.
The result follows.
\end{proof}

\begin{figure}
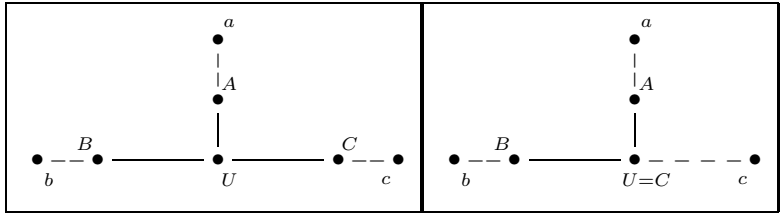

\[ \fbox{ \xygraph{
!{<0cm,0cm>;<.8cm,0cm>:<0cm,.8cm>::} 
!{(9,16) }*+{\bullet}="a"   !{(9.2,16.2) }*+{{}^a}="an"
!{(9,15) }*+{\bullet}="A"   !{(9.2,15.2) }*+{{}^{A}}="An"
!{(9,14) }*+{\bullet}="U"   !{(9.2,13.6) }*+{{}^{U}}="Un"
!{(7,14) }*+{\bullet}="B"   !{(6.8,14.2) }*+{{}^{B}}="Cn"
!{(11,14) }*+{\bullet}="C"  !{(11.2,14.2) }*+{{}^{C}}="Dn"
!{(6,14) }*+{\bullet}="b"   !{(6.2,13.6) }*+{{}^b}="bn"
!{(12,14) }*+{\bullet}="c"  !{(11.8,13.6) }*+{{}^c}="cn"
"a"-@{--}"A" "A"-"U" "U"-"B" "U"-"C" "B"-@{--}"b" "C"-@{--}"c"
} }
\fbox{ \xygraph{
!{<0cm,0cm>;<.8cm,0cm>:<0cm,.8cm>::} 
!{(9,16) }*+{\bullet}="a"   !{(9.2,16.2) }*+{{}^a}="an"
!{(9,15) }*+{\bullet}="A"   !{(9.2,15.2) }*+{{}^{A}}="An"
!{(9,14) }*+{\bullet}="U"   !{(9.2,13.6) }*+{{}^{U=C}}="Un"
!{(7,14) }*+{\bullet}="B"   !{(6.8,14.2) }*+{{}^{B}}="Cn"
!{(6,14) }*+{\bullet}="b"   !{(6.2,13.6) }*+{{}^b}="bn"
!{(11,14) }*+{\bullet}="c"  !{(10.8,13.6) }*+{{}^c}="cn"
"a"-@{--}"A" "A"-"U" "U"-"B"  "B"-@{--}"b" "U"-@{--}"c"
} }
\]
\caption{ The disposition of maximal orders and ends in the proof of Lemma \ref{l42}. Some continuous lines might be reduced to points, e.g., $C=U$ is possible, as on the left.}
\label{figmob}
\end{figure}

In order to prove Theorem \ref{thc},  we need to recall a few facts from the theory of 
representation by spinor genera. See  \cite{abelianos} or \cite{cyclic} for details.
Let $\Da$ be a full-rank order in a central simple algebra $\algeb$ over a number field $E$. We define the adelizations
 as follows:
\begin{enumerate}
\item $\Da_\ad=\prod_\wp\Da_\wp$ with the product topology, where the product is taken over all completions of $E$,
setting $\Da_\wp=\algeb_\wp$ when $\wp$ is archimedean.
\item $\algeb_{\ad}=\left\{a\in\prod_\wp\algeb_\wp\Big|a_\wp\in\Da_\wp
\textnormal{ for all but a finite number of  places }\wp\right\}$
with the only topology making each translation $x\mapsto a+x$ an open embedding of $\Da_\ad$, deffined above,
into $\algeb_{\ad}$.
\end{enumerate}
The group of units $\algeb_{\ad}^*$ acts on the set of full orders in $\algeb$. In fact, for any order $\Da$ and
every invertible adelic element $a\in\algeb_{\ad}^*$, there is a unique full order $\Da'$ satisfying $\Da'_\wp=
a_\wp\Da_\wp a^{-1}_\wp$ at every place $\wp$. We write $\Da'=a\Da a^{-1}$.
 With this definition, a genus is an $\algeb_{\ad}^*$-orbit,
while a class is an orbit under $\algeb^*$, diagonally identified with a subset of $\algeb_{\ad}^*$.

For any genus 
$\mathbb{O}=\mathrm{gen}(\Da)$ of full-rank orders, the spinor class field is, by definition, the class field $\Sigma$ 
corresponding to the class group  $E^*H(\Da)$, where $H(\Da)$ is the
group of reduced norms of adelic elements in $\algeb_\ad^*$ stabilizing $\Da$ by conjugation. If 
$\mathrm{dim}_E\algeb>4$ or $\algeb$ is an indefinite quaternion algebra, this fields classify 
conjugacy classes in the genus $\mathbb{O}$ in the sense that there exists an explicit map $\rho:\mathbb{O}\times\mathbb{O}\rightarrow\mathrm{Gal}(\Sigma/E)$ which is trivial precisely on pairs of isomorphic orders.
For any suborder $\Ha\subseteq\Da$ we define $\Phi=\Phi(\Da|\Ha)=\{\rho(\Da,\Da')|\Ha\subseteq\Da'\}$.
We let $\Delta$ be the group generated by $\Phi$, and we define $\Gamma=
\{\gamma\in\mathrm{Gal}(\Sigma/E)|\gamma\Phi=\Phi\}$.  The fixed fields $\Sigma^\Delta$ and $\Sigma^\Gamma$
are called the lower  representation field and upper representation field, respectively. 
When  $\Phi$ is a group, both fields coincide, and this common field is called the representation field. This is the case
if $\Ha$ is commutative and $\Da$ is maximal \cite[Theorem 1.1]{abelianos}.

To prove that a field $M\supseteq E$ is contained in $\Sigma^\Delta$ one simply proves that
the local spinor image
$$H_\wp(\Da|\Ha)=\{N(a)|a\in\algeb_\wp, \Ha_\wp\subseteq a_\wp\Da a_\wp^{-1}\}$$
is contained in the adelic class group $H_M=F^*N_{M/F}(J_M)$ for every finite or infinite place $\wp$.
At infinite places, this just mean that the algebra ramifies at every place where the extension $M/F$ does.
To prove that $M\supseteq\Sigma^\Gamma$, it suffices to prove that
the local norm $N_{M/F}(M\otimes_EE_\wp)$ is contained in 
$E^*\prod_{\mathfrak{q}}H_{\mathfrak{q}}(\Ha|\Ha)$, defined as above, for every local place $\wp$,
since clearly $H_{\mathfrak{q}}(\Ha|\Ha)H_{\mathfrak{q}}(\Da|\Ha)=H_{\mathfrak{q}}(\Da|\Ha)$
at every local place $\mathfrak{q}$.

\subparagraph{Proof of Theorem \ref{thc}:} 
Let  $N=\tilde{N}M$, where $\tilde{N}$ is a ramified quadratic extension of $E$ that splits at both $\wp$ and $\wp'$.
 By the condition that $M/E$ is inert at $\wp$ and $\wp'$, we conclude that $M$ embeds into $B$,
 and therefore there is an embedding $\phi:N\rightarrow\matrici_2(B)$ for which $M'=\phi(M)$ is contained in
the ring $B\mathfrak{1}$ of diagonal matrices with equal diagonal entries, which we identify with $B$ as before.
 We can also assume $\phi(\tilde{N})\subseteq\matrici_2(E)$, so $\phi(N)\subseteq\matrici_2(M')$.
Note that $\phi|_{\tilde{N}}$ can be diagonalized
at either place $\wp$ or $\wp'$, i.e., we can choose coordinates locally in a way
that the embedding is diagonal. Choose such diagonalization for each place $\mathfrak{q}\in\{\wp,\wp'\}$.
Note that this change of variables can be seen as a conjugation in $\matrici_2(E_\mathfrak{q})$,
 so we still have $M'_\mathfrak{q}=\phi(M_\mathfrak{q})\subseteq B_\mathfrak{q}$.
We let $\Ha\subseteq \phi(N)$ be maximal outside  $\{\wp,\wp'\}$, while we set
$$\Ha_{\mathfrak{q}}=\left\{\lbmatrix {l_1}00{l_2}\in \phi(N)\Big|l_1-l_2\in\mathfrak{m}_{M'_\mathfrak{q}}\right\},
\qquad \mathfrak{q}\in\{\wp,\wp'\}.$$
The latter condition implies that the residual algebra of $\Ha$ with respect to the maximal order
$\matrici_2\left(\oink_{B_q}\right)$ is
$$\mathbb{H}=\left[\Ha+\matrici_2\left(\mathfrak{m}_{B_\mathfrak{q}}\right)\right]
/\matrici_2\left(\mathfrak{m}_{B_\mathfrak{q}}\right)
\subseteq\left[\oink_{B_\mathfrak{q}}+
\matrici_2\left(\mathfrak{m}_{B_\mathfrak{q}}\right)\right]/
\matrici_2\left(\mathfrak{m}_{B_\mathfrak{q}}\right)
=\mathbb{B},$$ where $\mathbb{B}$ is the residual field of $B_\mathfrak{q}$,  and therefore the vertex
$\matrici_2\left(\oink_{B_\mathfrak{q}}\right)$ is a node in $\mathfrak{S}(\Ha)$.
Note that for $\mathfrak{q}\in\{\wp,\wp'\}$, the local order $\Ha_\mathfrak{q}$ is contained in
infinitely many maximal orders, since it is invariant under conjugation by the diagonal matrix
$\mathrm{diag}(\pi_{E_\mathfrak{q}},1)$, which act on the ball tree as multiplication by the uniformizer 
$\pi_{E_\mathfrak{q}}$. On the other hand $\oink_{M_\mathfrak{q}}\subseteq\Ha_\mathfrak{q}$,
whence $\mathfrak{S}(\Ha_\mathfrak{q})\subseteq\mathfrak{S}(\oink_{M_\mathfrak{q}})$. 

 Let $\Da_0$ be a maximal order in $\matrici_2(B)$ containing $\Ha$, so that $\Da_{0\mathfrak{q}}$ corresponds
 to a node of $\mathfrak{S}(\Ha_\mathfrak{q})$  for $\mathfrak{q}\in\{\wp,\wp'\}$.
 Now let $\Da$ be the full-rank order defined locally by $\Da_{\mathfrak{q}}=\Da_{0\mathfrak{q}}$ at all
finite places  $\mathfrak{q}\notin\{\wp,\wp'\}$, while $\Da_{\mathfrak{q}}=\Da_{0\mathfrak{q}}^{[1]}$
for $\mathfrak{q}\in\{\wp,\wp'\}$. By Lemma \ref{blemma}, this order is the intersection of all neighbors
of $\Da_{0\mathfrak{q}}$ at either place. We conclude that $\Da$ is the intersection of a finite family
of maximal orders.

Assume we replace $\mathfrak{A}$ by a matrix algebra $\mathfrak{A}'$
 and $\Da_0$ by a maximal order $\Da'_0$ in $\mathfrak{A}'$
containing a copy $\widetilde{\Ha}$ of $\oink_N$.
 It follows from \cite[Corollary 2]{abelianos} that $M$ is the representation  field for the pair
$(\Da'_0|\widetilde{\Ha})$, and therefore $H_{\mathfrak{q}}(\Da'_0|\widetilde{\Ha})\subseteq
E^*N_{M/E}(J_M)$ at every place $\mathfrak{q}$. 
Now we make three important remarks:
\begin{enumerate}
\item The computation of $\Phi_{\mathfrak{q}}=H_{\mathfrak{q}}(\Da'_0|\widetilde{\Ha})$ is local.
\item   The generated group $\langle \Phi_{\mathfrak{q}}\rangle$
 is independent of the maximal order  $\Da'_0\supseteq\widetilde{\Ha}$.
\item $\langle \Phi_{\mathfrak{q}}\rangle$
 is independent of the chosen copy $\widetilde{\Ha}$ of $\oink_N$.
\end{enumerate}
 The second statement holds since replacing $\Da'_0$ by a conjugate would replace $\Phi_{\mathfrak{q}}$
 by a set of the form $f\Phi_{\mathfrak{q}}$
with $f\in \Phi_{\mathfrak{q}}$, which generates the same group as 
$\Phi_{\mathfrak{q}}$, since the lattter contains the identity. 
We conclude that  
\begin{equation}\label{eqq}
\langle H_{\mathfrak{q}}(\Da|\Ha)\rangle = \langle H_{\mathfrak{q}}(\Da_0|\Ha)\rangle=
\langle H_{\mathfrak{q}}(\Da'_0|\widetilde{\Ha})\rangle\subseteq E^*N_{M/E}(J_M)
\end{equation}
at all places $\mathfrak{q}$,
except maybe $\wp$ and $\wp'$. Next we prove that the same holds at these two places. 
Theorem 1 implies that the local maximal orders $\Da_1$ for which 
$\oink_{M'_\mathfrak{q}}\subseteq\Da_1^{[1]}$ 
are the nodes of 
$\hat{\mathfrak{T}}(M'_\mathfrak{q})$, as $M'_\mathfrak{q}$
 is its own centralizer in the quaternion algebra $B_{\mathfrak{q}}$.
On the other hand, the distance between every pair of nodes in 
$\hat{\mathfrak{T}}(M'_\mathfrak{q})\cong\mathfrak{T}(M'_{\mathfrak{q}})/2$
  is always even. This implies that
$H_{\mathfrak{q}}(\Da|\Ha)\subseteq \oink_{\mathfrak{q}}^*E_{\mathfrak{q}}^{*2}$ at 
$\mathfrak{q}\in\{\wp,\wp'\}$, and the contention $M\subseteq\Sigma^\Delta$ follows since $M/E$ is unramified.

For the opposite contention, for $\mathfrak{q}\in\{\wp,\wp'\}$,
we observe that $\Ha_{\mathfrak{q}}$ is  invariant under conjugation  by
$\mathrm{diag}(\pi_{E_\mathfrak{q}},1)$, and every unit of the form $\mathrm{diag}(u_1,u_2)$,
with $u_1,u_2\in\oink_{M_{\mathfrak{q}}}^*$. This impplies the contention
$H_{\mathfrak{q}}(\Ha|\Ha)\supseteq\oink_{\mathfrak{q}}^*E_{\mathfrak{q}}^{*2}$.
 Note that $\oink_{\mathfrak{q}}^*E_{\mathfrak{q}}^{*2}$
is the local norm group for the extension $M/E$ at both places, as they are inert by hypotheses.
The corresponding contention at all other places follows from equation (\ref{eqq})
and \cite[Corollary 2]{abelianos}, as before.
\qed

\section{On roses}\label{sec4}

Let $L\subseteq B$ be a subalgebra, and let $C\subseteq B$ be its centralizer as in \S1. We define the
$r$-residual rings $\mathbb{B}_r=\oink_B/\mathfrak{m}_B^r$, $\mathbb{L}_r=(\oink_L+\mathfrak{m}_B^r)/\mathfrak{m}_B^r$,
and the centralizer $\mathbb{C}_r=C_{\mathbb{B}_r}(\mathbb{L}_r) $.
 We define the full
rose $\mathfrak{FR}(L)$ of $L$ as a rose with a level function defined on its vertices as follows:
\begin{itemize} \item There is one vertex $v$, namely the center,
 of level $0$. \item The vertices of level $r$, for $0<r\leq e'$,
 are all elements of $\mathbb{C}_r$. \item A vertex $c$ of level $r$ is joint by an edge to
its reduction modulo 
 $\mathfrak{m}_B^{r-1}$. \item All vertices of level $1$ are joined to $v$.
\item There are no other edges or vertices in $\mathfrak{FR}(L)$.
\end{itemize}
 It is apparent from the proof of Theorem \ref{tha} that this is, up to isomorphism,
the largest subgraph of $\mathfrak{S}(\oink_L)$ whose vertices corresponds to balls $U$ satisfying
$\rho(\pi_B)^{e'} \leq\rho(U)\leq1$. This motivates us to define the restricted rose $\mathfrak{RR}(L)$
as the connected component of $v$ in the graph obtained by removing from the full rose all edges joining  
$v$ to a vertex of level 1 corresponding to the class of an element in $C$.
All vertices of level $r$ in $\mathfrak{RR}(L)$, for $0<r\leq e'$, are the elements in the centralizer
$\mathbb{C}_r$ that are not congruent, modulo $\mathfrak{m}_B$ to an element
of $\oink_C$. This is the rose referred to in Theorem \ref{tha},
as follows easily from the proof of that theorem and Lemma \ref{l33}.
 The purpose on this section is to provide a precise description of the 
full rose, which determines the restricted rose as described above.

\begin{lemma}\label{thd}
Each element of the residue field $\mathbb{B}$ can be
lifted to an element of $\mathbb{C}_{e'}$.
\end{lemma}

\begin{proof}
 Set $L=K[F,\pi_L]$, where $F/K$ is an unramified field extension,
and $\rho(\pi_L)=\rho(\pi_B)^{e'}$.
Note that the image of $\oink_L$ in $\mathbb{B}_{e'}$ coincide with the image of $\oink_F$, whence the full
rose  $\mathfrak{FR}(L)$ coincide with the part of the graph $\mathfrak{S}(\oink_F)$ corresponding to balls
whose radii $r$ satisfy $\rho(\pi_B)^{e'}\leq r \leq1$. It suffices to prove that
$\mathfrak{S}(\oink_F)$ has a vertex corresponding to a ball of radius $\rho(\pi_B^{e'})=\rho(\pi_L)$ with a center in 
each residual class in $\mathbb{B}$. This is immediate from Lemma \ref{l33} and
the proof of Theorem \ref{tha}, since the centralizer 
of $F$ contains a maximal unramified subfield of $B$.
\end{proof}

\begin{theorem}\label{thd}
Let $B$ be a $n^2$-dimensional central division algebra over a local field $K$.
 Let $L\subseteq B$ be simple subalgebra.
Let $f=f(L/K)$ be the residual degree, let $e=e(L/K)$ be the ramification index, 
and let $e'=\frac{n}{e}$. 
The full rose $\mathfrak{FR}(L)$ of $L$ is the largest possible $(q^n-q^{e'},e')$-rose having 
nodes only at vertices whose distance to the center is divisible by $f$.
\end{theorem}

\begin{proof}
 Let $C\subseteq B$ and $\mathbb{C}_r\subseteq\mathbb{B}_r$ be as before.
The full rose has two symmetries that can be exploited to compute it. The first one comes from the translational 
symmetry of the ring $\mathbb{C}_r$. If a particular element $\bar{\epsilon}$ of $\mathbb{C}_1$ can be lifted to an element $\epsilon\in\mathbb{C}_r$, the set of all its possible liftings has the form $\epsilon+U_r$ where 
$U_r=\mathrm{ker}(\mathbb{C}_r\rightarrow\mathbb{C}_1)$
is the set of liftings of $\bar{0}\in\mathbb{C}_1$ to $\mathbb{C}_r$. This means that the corresponding 
branches in the rose look similar up to the level $r$. By
Lemma \ref{thd} we can take $r=e'$  (See Figure 5(ii)).
On the other hand, we know
from the proof of Theorem 1.1 that the group $\mathrm{GL}_2(C)$, the centralizer of $L^*$
in $\mathrm{GL}_2(B)$, acts transitively on the nodes of  $\hat{\mathfrak{T}}(C)$. As
$\mathrm{GL}_2(C)$ commutes with $\oink_L$, these conjugations are symmetries of
$\mathfrak{S}(\oink_L)$. This gives us a vertical 
symmetry in the branch below the ball $V_0^{[0]}=\oink_B$ in Figure 5(i)
 that can be use to determine its shape at lower levels in terms of the
higher ones. Note that $V_0^{[f]}$ is the first node of $\hat{\mathfrak{T}}(C)$
 after $V_0^{[0]}$ in the ray joining $V_0^{[0]}$
with the end $0$, as we saw in the proof Theorem \ref{tha} that every vertex between 
$V_0^{[0]}$ and $V_0^{[f]}$ is a bridge. 
This impplies, by the simmetry in Figure 5(ii), 
that the full rose cannot have a node whose level is smaller than $f$,
while all vertices at level $f$ are nodes. Now we apply the simmetry in Figure 5(i)
to prove that there are no nodes between the levels $f$ and $2f$, while each vertex of level $2f$ is a node. 
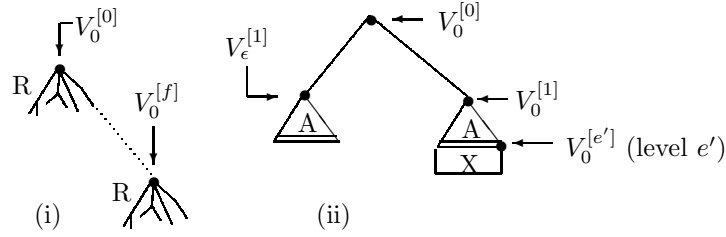
\begin{figure}
\unitlength 1mm 
\linethickness{0.4pt}
\ifx\plotpoint\undefined\newsavebox{\plotpoint}\fi 
\begin{picture}(83.45,29.75)(30,0)
\multiput(39.75,22.75)(-.033482143,-.053571429){112}{\line(0,-1){.053571429}}
\put(40.25,22.75){\line(-1,0){.5}}
\multiput(72.25,19)(-.033482143,-.053571429){112}{\line(0,-1){.053571429}}
\put(68.5,13){\line(1,0){8.75}}
\put(77.25,13){\line(-3,4){4.5}}
\put(72.75,19){\line(-1,0){.5}}
\put(73,15.5){\makebox(0,0)[cc]{A}}
\put(90.25,12.25){\line(1,0){8.75}}
\put(99,12.25){\line(-3,4){4.5}}
\put(94.5,18.25){\line(-1,0){.5}}
\put(94.75,14.75){\makebox(0,0)[cc]{A}}
\put(72.1,19){\makebox(0,0)[cc]{ $\bullet$}}
\put(98.2,12.3){\makebox(0,0)[cc]{ $\bullet$}}
\put(93.8,18.25){\makebox(0,0)[cc]{ $\bullet$}}
\put(39.5,22.5){\makebox(0,0)[cc]{ $\bullet$}}
\put(52,7.5){\makebox(0,0)[cc]{ $\bullet$}}
\put(65,26){\makebox(0,0)[cc]{$V_{\epsilon}^{[1]}$}}
\put(65,23){\line(0,-1){4}}
\put(65,19){\vector(1,0){4}}
\put(81.5,29){\makebox(0,0)[cc]{$\bullet$}}
\put(98.75,12){\line(0,-1){3.25}}
\put(98.75,8.75){\line(-1,0){8.75}}
\put(90,8.75){\line(0,1){3.25}}
\put(94.5,10){\makebox(0,0)[cc]{X}}
\put(45,28.5){\makebox(0,0)[cc]{$V_0^{[0]}$}}
\put(53,18.5){\makebox(0,0)[cc]{$V_0^{[f]}$}}
\put(41.75,28.75){\line(-1,0){1.55}}
\put(40,28.75){\vector(0,-1){4.25}}
\put(38.5,3){\makebox(0,0)[cc]{(i)}}
\put(76.5,3){\makebox(0,0)[cc]{(ii)}}
\put(90.5,13){\line(1,0){8.25}}
\put(69,13.75){\line(1,0){7.75}}
\put(103.6,18.7){\makebox(0,0)[cc]{$V_0^{[1]}$}}
\put(99.6,18.7){\vector(-1,0){4}}
\put(35,20.5){\makebox(0,0)[cc]{R}}
\put(48.25,6.25){\makebox(0,0)[cc]{R}}
\put(93.25,29.25){\makebox(0,0)[cc]{$V_0^{[0]}$}}
\put(88.25,29.25){\vector(-1,0){4}}
\put(105.5,12.75){\vector(-1,0){5.25}}
\put(118,12.5){\makebox(0,0)[cc]{$V_0^{[e']}$ (level $e'$)}}
\put(37,18){\line(0,-1){1.55}}
\multiput(40,22.5)(.03333333,-.07333333){75}{\line(0,-1){.07333333}}
\multiput(40,22)(-.03125,-.3125){8}{\line(0,-1){.3125}}
\multiput(39.75,19.5)(.03289474,-.05263158){38}{\line(0,-1){.05263158}}
\multiput(39.75,19.25)(-.03947368,-.03289474){38}{\line(-1,0){.03947368}}
\multiput(40,22.75)(.03658537,-.03353659){82}{\line(1,0){.03658537}}
\multiput(43,20)(.03333333,-.05){45}{\line(0,-1){.05}}
\put(49.5,3){\line(0,-1){1.55}}
\multiput(52.25,7.75)(-.033482143,-.053571429){112}{\line(0,-1){.053571429}}
\multiput(52.5,7.5)(.03333333,-.07333333){75}{\line(0,-1){.07333333}}
\multiput(52.5,7)(-.03125,-.3125){8}{\line(0,-1){.3125}}
\multiput(52.25,4.5)(.03289474,-.05263158){38}{\line(0,-1){.05263158}}
\multiput(52.25,4.25)(-.03947368,-.03289474){38}{\line(-1,0){.03947368}}
\multiput(52.5,7.75)(.03658537,-.03353659){82}{\line(1,0){.03658537}}
\multiput(55.5,5)(.03333333,-.05){45}{\line(0,-1){.05}}
\put(52.5,15){\vector(0,-1){5}}
\multiput(44.43,17.93)(.55,-.66667){16}{{\rule{.4pt}{.4pt}}}
\multiput(94,18.25)(-.033482143,-.053571429){112}{\line(0,-1){.053571429}}
\multiput(72.5,19)(.0336538462,.0413461538){260}{\line(0,1){.0413461538}}
\multiput(81.25,29.75)(.0407523511,-.0336990596){319}{\line(1,0){.0407523511}}
\end{picture}
\caption{The two simetries of the full rose. Here $V_a^{[t]}$ is the ball of radius
$\rho(\pi)^t$ around $a\in B$.}
\end{figure}
Iterating this process, the result follows.
\end{proof}

\begin{example}
If $e'\leq f$, then the only node of the full rose is its center. 
\end{example}

\begin{example}
Let $B=K[F,\pi_B]$ be a central division algebra of dimension $8^2$, where
conjugation by the uniformizer $\pi_B$ induces a generator of the Galois group
of the unramified degree-8 extension $F/K$. Let
 $L=K[E,\pi_B^4]$ with $E\subseteq F$ a quadratic extension of $K$. Then $L$ is
a partially ramified quartic extension of $K$. In this case
$f=2$, while $e'=4$. We conclude that every point at level $2$ in the full rose is a node.
\end{example}

\section{Further applications and examples}

\subparagraph{Proof of Theorem \ref{tho}}
For the first statement, assume that
 the maximal power of $2$ dividing $f(L/K)$ fails to divide $n=\sqrt{\dim_KB}$. We conclude
that the residue field $\mathbb{L}$ of $L$  does not embed into $\mathbb{B}$. This impplies that the representation
of $\mathbb{L}$ defined by the natural inclusion $\phi:\mathbb{L}\hookrightarrow\matrici_2(\mathbb{B})$
is irreducible, whence $\mathfrak{S}(\oink_L)$ contains an isolated point, 
and therefore consists exactly of one point.

For the second statement, we observe that  $F$ embeds into $B$, 
so it can be assumed to be a subfield of $B$. As such, it is its own centralizer,
so Theorem \ref{tha} gives us a description of the branch $\mathfrak{S}(\oink_F)$ which contains
$\mathfrak{S}(\oink_L)$. Furthermore, note that $L=F[\pi_L]$ where $\pi_L$ 
satisfies an equation over $F$ of the form
$x^2+\epsilon\pi_Fx+\pi_F=0$, for some $\epsilon\in\oink_F$,
 whence its image in $\matrici_2(B)$ belongs to the centralizer
$\matrici_2(F)$ of $F$, and by an appropiate choice of basis it can be assumed to be of the form
$\phi(\pi_L)=\lbmatrix 0{-\pi_F}1{-\epsilon\pi_F}\in\matrici_2(F)$. Therefore, the branch $\mathfrak{S}(\oink_L)$
must be invariant under conjugation by this matrix.

In order to describe how this matrix acts by conjugation on maximal orders of the form $\Da_U$, where $U$ is
a ball, we observe that $$\phi(\pi_L)\left(\begin{array}c z\\1\end{array}\right)\oink_B=\left(\begin{array}c w\\1\end{array}\right)\oink_B u \textnormal{, where } u=z-\epsilon\pi_F,
w=-\pi_Fu^{-1}.$$
This implies that the corresponding 
action on balls takes the triplet of ends $(\infty,1,0)$ of the ball-tree onto $(0,u'\pi_F,\epsilon^{-1})$,
where $u'=-(1-\epsilon\pi_F)^{-1}$ is a unit. These ends, and their respective images,
 are located as shown in Figure 6. Since every vextex between
$V_0^{[0]}$ and $V_0^{[e(B/F)]}$ is a bridge in $\mathfrak{S}(\oink_F)$, it suffices to show that 
$V_0^{[0]}$ and $V_0^{[e(B/F)]}$ have valency one in $\mathfrak{S}(\oink_L)$. As conjugation by 
$\phi(\pi_L)$ defines a symmetry interchanging them,
it suffices to prove it for $V_0^{[0]}$, but this follows easily from the fact that the image
$\mathbb{H}$ of $\Ha=\oink_L$ in $\matrici_2(\mathbb{B})$ contains the non-trivial nilpotent
element $\overline{\phi(\pi_L)}=\lbmatrix0010$.

\begin{figure}
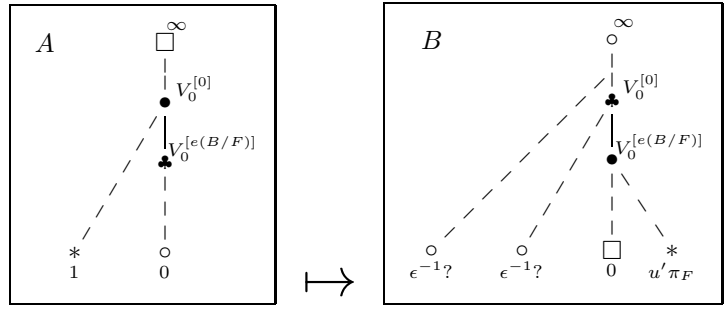

\[ \fbox{ \xygraph{
!{<0cm,0cm>;<.8cm,0cm>:<0cm,.8cm>::} 
!{(7,16) }*+{A}="name"
!{(9,16) }*+{\square}="a"   !{(9.2,16.2) }*+{{}^\infty}="an"
!{(9,15) }*+{\bullet}="A"   !{(9.5,15.2) }*+{{}^{V_0^{[0]}}}="An"
!{(9,14) }*+{\clubs}="U"   !{(9.8,14.2) }*+{{}^{V_0^{[e(B/F)]}}}="Un"
!{(7.5,12.5) }*+{*}="b"   !{(7.5,12.1) }*+{{}^1}="bn"
!{(9,12.5) }*+{\circ}="c"  !{(9,12.1) }*+{{}^0}="cn"
"a"-@{--}"A" "A"-"U"  "A"-@{--}"b" "U"-@{--}"c"
} }\quad\textnormal{\Huge$\mapsto$\normalsize}\quad
\fbox{ \xygraph{
!{<0cm,0cm>;<.8cm,0cm>:<0cm,.8cm>::} 
!{(6,16) }*+{B}="name"
!{(9,16) }*+{\circ}="a"   !{(9.2,16.2) }*+{{}^\infty}="an"
!{(9,15.5) }*+{}="E"  
!{(9,15) }*+{\clubs}="A"   !{(9.5,15.2) }*+{{}^{V_0^{[0]}}}="An"
!{(9,14) }*+{\bullet}="U"   !{(9.8,14.2) }*+{{}^{V_0^{[e(B/F)]}}}="Un"
!{(10,12.5) }*+{*}="b"   !{(10,12.1) }*+{{}^{u'\pi_F}}="bn"
!{(7.5,12.5) }*+{\circ}="d"   !{(7.5,12.1) }*+{{}^{\epsilon^{-1}?}}="dn"
!{(6,12.5) }*+{\circ}="e"   !{(6,12.1) }*+{{}^{\epsilon^{-1}?}}="en"
!{(9,12.5) }*+{\square}="c"  !{(9,12.1) }*+{{}^0}="cn"
"a"-@{--}"A" "A"-"U"  "U"-@{--}"b" "U"-@{--}"c"  "A"-@{--}"d" "E"-@{--}"e" 
} }
\]
\caption{ The disposition of maximal orders and ends in the proof of Theorem \ref{tho}.
The same symbol denotes a vertex, or end, in Figure 6A, and its corresponding image under conjugation by
$\phi(\pi_L)$ in Figure 6B.  The precise
position of the end $\epsilon^{-1}$ depends on $\epsilon\in\oink_F$, if $\epsilon=0$ we
have $\epsilon^{-1}=\infty$, but this is irrelevant to
determine the position of the image path.}
\label{figmob}
\end{figure}

For the final statement, we note that if $L$ is not a field then it contains a non-trivial idempotent,
which can be assumed to be $\lbmatrix1000$. This element is contained precisely in the maximal orders of the form
$\Da_t=\lbmatrix{\oink_B}{\pi_B^t\oink_B}{\pi_B^{-t}\oink_B}{\oink_B}$. On the other hand, $L$ is contained
in the centralizer $C=\lbmatrix B00B$ of $\lbmatrix1000$, and the set of integers  in $C$ is contained
in each order $\Da_t$.  The result follows.\qed

\begin{example}
We end this work by giving a few examples of small admisible branches.
First observe that the family of admissible branches is closed under non-empty intersections, 
since an order contains a particular suborder $\Ha$ if and only if
it contains each element in a generating set of $\Ha$, while a set generates an order if and only 
if it is contained in at least one maximal order.
  In the pictures we assume that $B$ is the unique quaternion division 
algebra over $\Qt_2$:
\begin{enumerate}
\item Figure 7A-B show two intersections of the branch $\mathfrak{S}(\oink_B)$
and a ball centered at the white circle, which is a bridge (Fig. 7A) or a node (Fig. 7B).
\item the flower in Figure 7C is obtained by intersecting the branch $\mathfrak{S}(\oink_B)$ with 
two balls of radius 2 centered at the two white circles.
\item Figure 7D can be obtained intersecting the branch in Figure 7B with a conjugate (c.f. Figure 7B') replacing
 the neighbor of valency 2 at the left of the center with a neighbor on valency 1 (a leaf),
using the $3$-transitivity on ends.
\end{enumerate}

On the other hand, we claim that the branch $\mathfrak{S}$ in Figure 7E is not admissible. 
In order to prove this we choose ends beyond them as in the figure. Note that by the transitivity on triples of ends
 we can always assume that three of the ends are $\infty$, $0$, and $1$. Furthermore,
 it is immediate that the end $a$ is a unit in $\oink_B$ that is not congruent to 
$1$ modulo $\pi_B$, while $1+a$ lies as shown in the picture since is not congruent to either $1$, $0$, or $a$.
Note also that the white circle is the vertex $\Da_0$ corresponding to the ball $V_{0}^{[0]}$
of radius $1$ centered at $0$. Now, let $V_t^{[2]}=t+\pi_B^2\oink_B$ for $t\in\oink_B$ as before, while 
$V_\infty^{[2]}=\pi_B^{[-2]}\oink_B$. Write $\Da_{*,t}=\Da_{V_t^{[2]}}$, for the corresponding maximal
order,  for each $t\in\{a,1+a,1,0,\infty\}$.
Since $\mathfrak{S}=\mathfrak{S}(\Ha_{\mathfrak{S}})$ by the hypotheses that $\mathfrak{S}$ is admissible,
there is a matrix $\eta\in\Ha_{\mathfrak{S}}$ such that $\eta\notin\Da_{*,1+a}$.
However, the condition $\eta\in\Ha_{\mathfrak{S}}$ implies
$\eta\in\Da_{*,t}$ , for each $t\in\{a,1,0,\infty\}$. We claim this cannot be the case.

Note that $\eta$ has integral coefficients since $\eta\in\Da_0$.
The condition $\eta\in\Da_{*,\infty}\cap\Da_{*,0}$ implies that $\eta\equiv\lbmatrix x00y$
modulo $\pi_B^2$ for some $x,y\in\oink_B$.
The condition $\eta\in\Da_{*,a}$ implies $xa\equiv ay\ (\mathrm{mod}\ \pi_B^2)$, since $\eta$ satisfies
$\eta\Lambda_{*,a}\subseteq \Lambda_{*,a}$, where
$$\Lambda_{*,a}=\left(\begin{array}ca\\1\end{array}\right)\oink_B+
\left(\begin{array}c\pi_B^2\\0\end{array}\right)\oink_B.$$ Similarly,  the condition
$\eta\in\Da_{*,1}$ implies $x\equiv y\ (\mathrm{mod}\ \pi_B^2)$. The last two conditions imply
$x(1+a)=(1+a)y\ (\mathrm{mod}\ \pi_B^2)$ and the contradiction follows since this implies
$\eta\in\Da_{*,1+a}$ by the same reasoning.
\begin{figure}
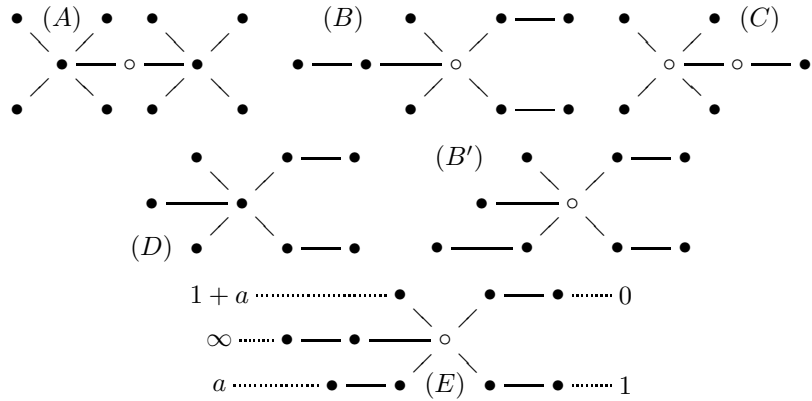

\[ \xygraph{
!{<0cm,0cm>;<.6cm,0cm>:<0cm,.6cm>::} 
!{(0,1) }*+{(A)}="A" 
!{(0,0) }*+{\bullet}="a" 
!{(-1,-1) }*+{\bullet}="b" 
!{(-1,1) }*+{\bullet}="c"
!{(1,-1) }*+{\bullet}="d" 
!{(1,1) }*+{\bullet}="e"
!{(1.5,0) }*+{\circ}="a1" 
!{(3,0) }*+{\bullet}="a0" 
!{(2,-1) }*+{\bullet}="b0" 
!{(2,1) }*+{\bullet}="c0"
!{(4,-1) }*+{\bullet}="d0" 
!{(4,1) }*+{\bullet}="e0"
"a"-"b" "a"-"c" "a"-"d" "a"-"e" "a"-"a1" "a1"-"a0" "a0"-"b0" "a0"-"c0" "a0"-"d0" "a0"-"e0"
} \quad
\xygraph{
!{<0cm,0cm>;<.6cm,0cm>:<0cm,.6cm>::} 
!{(0.5,1) }*+{(B)}="A" 
!{(1,0) }*+{\bullet}="a"  
!{(-.5,0) }*+{\bullet}="e"
!{(3,0) }*+{\circ}="a0" 
!{(2,-1) }*+{\bullet}="b0" 
!{(2,1) }*+{\bullet}="c0"
!{(4,-1) }*+{\bullet}="d0" 
!{(4,1) }*+{\bullet}="e0"
!{(5.5,-1) }*+{\bullet}="d1" 
!{(5.5,1) }*+{\bullet}="e1"
 "a"-"e" "a"-"a0" "a0"-"b0" "a0"-"c0" "a0"-"d0" "a0"-"e0" "d0"-"d1"
"e0"-"e1"
} \quad
\xygraph{
!{<0cm,0cm>;<.6cm,0cm>:<0cm,.6cm>::} 
!{(2,1) }*+{(C)}="A" 
!{(0,0) }*+{\circ}="a" 
!{(-1,-1) }*+{\bullet}="b" 
!{(-1,1) }*+{\bullet}="c"
!{(1,-1) }*+{\bullet}="d" 
!{(1,1) }*+{\bullet}="e"
!{(1.5,0) }*+{\circ}="a1" 
!{(3,0) }*+{\bullet}="a0" 
"a"-"b" "a"-"c" "a"-"d" "a"-"e" "a"-"a1" "a1"-"a0"
}
\]
\[  \xygraph{
!{<0cm,0cm>;<.6cm,0cm>:<0cm,.6cm>::} 
!{(1,-1) }*+{(D)}="A"
!{(1,0) }*+{\bullet}="a"  
!{(3,0) }*+{\bullet}="a0" 
!{(2,-1) }*+{\bullet}="b0" 
!{(2,1) }*+{\bullet}="c0"
!{(4,-1) }*+{\bullet}="d0" 
!{(4,1) }*+{\bullet}="e0"
!{(5.5,-1) }*+{\bullet}="d1" 
!{(5.5,1) }*+{\bullet}="e1"
 "a"-"a0" "a0"-"b0" "a0"-"c0" "a0"-"d0" "a0"-"e0" "d0"-"d1"
"e0"-"e1"
} \qquad
\xygraph{
!{<0cm,0cm>;<.6cm,0cm>:<0cm,.6cm>::} 
!{(0.5,1) }*+{(B')}="A" 
!{(0,-1) }*+{\bullet}="b1"  
!{(1,0) }*+{\bullet}="e"
!{(3,0) }*+{\circ}="a0" 
!{(2,-1) }*+{\bullet}="b0" 
!{(2,1) }*+{\bullet}="c0"
!{(4,-1) }*+{\bullet}="d0" 
!{(4,1) }*+{\bullet}="e0"
!{(5.5,-1) }*+{\bullet}="d1" 
!{(5.5,1) }*+{\bullet}="e1"
"a0"-"e"
 "b1"-"b0" "a0"-"b0" "a0"-"c0" "a0"-"d0" "a0"-"e0" "d0"-"d1"
"e0"-"e1"
} \]
\[
\xygraph{
!{<0cm,0cm>;<.6cm,0cm>:<0cm,.6cm>::} 
!{(3,-1) }*+{(E)}="A"
!{(1,0) }*+{\bullet}="a1"  
!{(-.5,0) }*+{\bullet}="a2"
!{(3,0) }*+{\circ}="a0" 
!{(.5,-1) }*+{\bullet}="b1"
!{(2,-1) }*+{\bullet}="b0" 
!{(2,1) }*+{\bullet}="c0"
!{(4,-1) }*+{\bullet}="d0" 
!{(4,1) }*+{\bullet}="e0"
!{(5.5,-1) }*+{\bullet}="d1" 
!{(5.5,1) }*+{\bullet}="e1"
!{(-2,0) }*+{\infty}="a3"
!{(-2,-1) }*+{a}="b3"
!{(-2,1) }*+{1+a}="c3"
!{(7,-1) }*+{1}="d3" 
!{(7,1) }*+{0}="e3"
 "a1"-"a2" "a1"-"a0" "a0"-"b0" "a0"-"c0" "a0"-"d0" "a0"-"e0" "d0"-"d1" "b0"-"b1"
"e0"-"e1" "a2"-@{.}"a3" "b1"-@{.}"b3" "c0"-@{.}"c3" "d1"-@{.}"d3" "e1"-@{.}"e3"
} 
\]
\caption{ A few admissible branches and a counterexample.}
\label{figmob}
\end{figure}
\end{example}

\section{Acnowledgements}
This first author was suported by Fondecyt, Grant No 1160603. The second author was supported by 
Fondecyt, Grant No 1120844.

\end{document}